\documentclass[a4paper]{article}

\usepackage{geometry}                
\usepackage{amsmath}                 
\usepackage{amssymb}                 
\usepackage{amsbsy}
\usepackage{amsthm}
\usepackage{amsfonts}

\usepackage{graphicx}
\usepackage{subfigure}
\usepackage{epstopdf,cite}

\usepackage[colorlinks,linkcolor=red]{hyperref} 

\newtheorem{lemma}{Lemma}[section]
\newtheorem{theorem}{Theorem}[section]

\newtheorem{definition}{Definition}[section]

\newtheorem{property}{Property}[section]                 

\begin{document}

\title{\textbf{Detection of bistable structures via the Conley index and applications to biological systems}}

\author{Junbo Jia$^{1,2}$, Pan Yang$^3$, Huaiping Zhu$^4$, Zhen Jin$^5$, Jinqiao Duan$^{6}$, Xinchu Fu$^{7,}$\thanks{Corresponding author. Tel: +86-21-66132664; Fax: +86-21-66133292; Email address: xcfu@shu.edu.cn} \\ \\
	$^1${\small Key Laboratory of Systems Biology, Hangzhou Institute for Advanced Study,}\\
		{\small  University of Chinese Academy of Sciences, Hangzhou 310024, China}\\
	$^2${\small Center for Excellence in Molecular Cell Science, Shanghai Institute of Biochemistry and Cell Biology,}\\
		{\small Chinese Academy of Sciences, Shanghai 200031, China}\\
	$^3${\small School of Mathematical Sciences, Changsha Normal University, Changsha, Hunan 410100, China}\\
	$^4${\small Laboratory of Mathematical Parallel Systems (LAMPS), Department of Mathematics and Statistics,}\\
		{\small York University, Toronto, ON, M3J 1P3, Canada}\\
	$^5${\small  Complex Systems Research Center, Shanxi University, Taiyuan, Shanxi 030051, China}\\
	$^6${\small Department of Applied Mathematics, Illinois Institute of Technology, Chicago, IL 60616, USA}\\
	$^7${\small Department of Mathematics, Shanghai University, Shanghai 200444, China}\\
}

\date{(\today)}


\maketitle

\begin{abstract}

\noindent Bistability is a ubiquitous phenomenon in life sciences. In this paper,
two kinds of bistable structures in dynamical systems are studied: One is two one-point attractors,
another is a one-point attractor accompanied by a cycle attractor. By the Conley index theory,
we prove that there exist other isolated invariant sets besides the two attractors, and also
obtain the possible components and their configuration. Moreover, we find that there is always
a separatrix or cycle separatrix, which separates the two attractors. Finally, the biological meanings
and implications of these structures are given and discussed.

\medskip

\end{abstract}

\section{Introduction}

Bistability is a widespread phenomenon in everyday life and life sciences. Here bistability means
that a dynamical system is in one of two stable states, and its state does
not change under a slight disturbance. A simple example of this kind of phenomenon is
the lamp switch, which has two mutually exclusive and stable states, `on' and `off.' The
memory unit in the electronic device, flip-flop circuit, which is always in one of two states,
`0' and `1', is also an example of bistability. Besides, interestingly, there are many examples
of such bistability in life-related fields ranging from microscopic gene expression to
macroscopic species competition. For example, at the molecular level, whether or not genes
in cells are expressed~\cite{Griffith1971}, the switching of different genes~\cite{tian2004bistability},
and the differentiation of cells in developmental epistemology~\cite{ferrell2012bistability},
can be considered as bistable phenomena. There are also many examples at the population level,
such as the outbreak of disease~\cite{Zhou2012}, the number of single species~\cite{Ludwig1978,Ludwig1997},
the interaction between different species~\cite{Freedman1986}, including competition,
symbiosis, predation.

The study of bistability is of considerable significance in decoding the mysteries of life
because this phenomenon can be served as a commonality shared by many life-related fields.
On the one hand, bistability helps to understand how the biological object switches between
different stable states so that we can assess conditions towards the desired state.
On the other hand, bistability is a bridge between monostability and multistability. It can
be regarded as a building block for the study of more complicated structures and dynamics.

Stable states may be described by attractors in dynamical systems. Typically, an attractor
refers to the whole asymptotically stable orbit in the dynamical system. So, as its name suggests,
the orbits neighboring the attractor will be attracted and will finally approach the attractor as
time goes on. Some asymptotically stable orbits, such as asymptotically stable equilibria,
asymptotically stable periodic orbits, and strange attractors, are attractors. However,
sometimes attractors also refer to the attracting invariant sets that are composed of whole
orbits. To avoid ambiguity, the attractor in this paper adopts the previous definition,
and the one-point attractor and the cycle attractor refer to the asymptotically stable equilibrium
and the asymptotically stable periodic orbit, respectively.

In this paper, two classes of dynamical systems with precisely two attractors are studied.
The first is the system with two one-point attractors, and the second is that with a one-point
attractor and a cycle attractor. They are the two most common bistable structures in
biological fields. With the help of the Conley index, we find that there exists another
invariant set except for these two attractors, and also obtained its Conley index. We also get
the possible compositions of this invariant set and their connecting structure with the
attractors. Moreover, we find that there is always a separatrix or cycle separatrix between two
attractors, which divide the region we considered into two sub-regions so that almost all the
orbits in different subregion flow to different attractors.

Conley index, which is named after Charles Conley, is the significant generalization of Morse
index~\cite{Conley1978,Smoller2012,Mischaikow2002}. In Morse index theory, the study
object is the hyperbolic equilibrium, and its Morse index is defined as the dimension of the
unstable manifold. However, in the Conley index theory, a hyperbolic equilibrium is
generalized as an isolated invariant set. The primary strategy of the Conley index is to find
a neighborhood that isolates the invariant set, then deduce the dynamical structure or the
properties inside by examining the behavior of flow at the boundary.

Conley index can be used as a topological tool for the studying of a dynamical system~\cite{Conley1978,Mischaikow2002}.
One characteristic of this theory is that it captures the stable property of a dynamical system.
In other words, according to its Continuation Property \ref{th-3} listed in Appendix \ref{sec-appendix},
even though the isolated invariant set can bifurcate to a different
structure, its Conley index would stay the same. So sometimes we cannot get the specific
structure of the invariant set only from its Conley index, but we can exclude some structure.
However, if we want to get a more refined structure about this invariant set, some additional
information becomes indispensable, which is the research idea of this paper.

The rest of the paper is organized as follows. In Section~\ref{sec-2}, two kinds of bistable
structures are studied, and the main results are presented as well. In Section~\ref{sec-3},
we present four examples and two applications that exactly have a bistable structure. Finally,
the conclusion and further discussion are given in Section~\ref{sec-4}.

\section{Main results}
\label{sec-2}

The research object of this paper is the bistable structure in a dynamical system, which can be
written as follows:
\begin{equation}\label{eq-2-1}
	\frac{d\mathbf{x}(t)}{dt}=\mathbf{f(x}(t)).
\end{equation}
Here $\mathbf{x}(t)$ is the state vector in $\mathbb{R}^n$, and the
vector-valued function $\mathbf{f(x}(t))$ is differentiable in its domain. In this section, we
consider this system two cases: two one-point attractors and coexistence of a cycle attractor and
a one-point attractor, since they are the two most common bistable structures in biological models.

\subsection{Case 1: Two one-point attractors}
\label{sec-2-1}

In this case, we first make the following assumption:

\begin{description}
	\item[\textbf{H1}] There exists a bounded and closed region $U$ in the domain of system (\ref{eq-2-1}), such that
	there are precisely two asymptotically stable equilibria in the interior of $U$, say $A_1$ and
	$A_2$, and on the boundary $\partial U$, all orbits run from outside into interior immediately.
\end{description}

Usually, researchers only care about the behaviors within some regions or the structure they
are interested in. For example, for biological models, only the quadrants where all variables
are non-negative are focused, since variables usually represent the number of population or
concentration, and only the positive quantities make sense. Sometimes whether the model
has a periodic orbit is concerned as well because this orbit can cause the variable to oscillate.
However, in this paper, we are only interested in the structure in assumption \textbf{H1},
and there are three main reasons. First, $U$ is an attracting region from the outside view of it,
though there are two stable equilibria within $U$. That makes it possible to bifurcate between
monostable and this kind of bistable structure. Second, the region with the bistable structure
as a whole can be considered as one of the attractors in another bistable structure. That is,
a bistable structure can be embedded into another bistable structure. As a result, we may
be able to use the bistable structure as the cornerstone to study the multi-stable structure.
Third, the whole region $U$ is attracting, which is based on the fact that, basically, the
number of biological objects can neither become negative nor increase indefinitely due to
limited environmental carrying capacity, such as species population, and densities of cells or
microorganism.

The main results are presented below.

\begin{theorem}\label{th-2-1}
	For system (\ref{eq-2-1}), if \textbf{H1} is satisfied, then we have the following conclusions:
	\par (i) There exists other isolated invariant set, denoted by $S$, besides of $A_1$ and $A_2$,
	and its Conley index is $\Sigma^1$.
	\par (ii) There also exists connecting orbits from the invariant set $S$ to attractors, $A_1$ and
	$A_2$.
\end{theorem}

\begin{proof}
	(i) Firstly, let us construct an isolating neighborhood for invariant set $S$. Since $A_i, i=1, 2$,
	is asymptotically stable equilibrium, there must exist two small open neighborhood $U_1$ and
	$U_2$ which are disjoint in $\mathrm{Int}(U)$, i.e., the interior of $U$, such that $A_i\in U_i$
	and all orbits that pass through the boundary of $U_i$ will gradually approach the equilibrium
	$A_i$ as time goes to infinity, $t\rightarrow \infty$. Then the region
	$N=U\setminus \{U_1\sqcup U_2\}$ is an isolating neighborhood, which is shown in
	Fig.~\ref{figure-1}.
	
	It is not difficult to find that the exit set $N_0$ of isolating neighborhood $N$ is the boundary
	of $U_1$ and $U_2$, namely, $N_0=\partial U_1 \sqcup \partial U_2$. Then by Definition \ref{def-1}
	of the index pair in Appendix~\ref{sec-appendix}, we can easily verify that $(N,N_0)$ is an index
	pair. By forming the cone over the exit set $N_0$, we can obtain that the Conley index of index
	pair $(N,N_0)$ is $\Sigma^1$ (see also Example 3 in~\cite{Kappos1996}). According to the
	Wa\.zewski Property~\ref{th-1} listed in Appendix~\ref{sec-appendix}, the
	interior of isolating neighborhood $N$ must contain isolated invariant set $S$.
	
	(ii) If we dig $U_2$ out from region $U$, then the remaining region and the corresponding exit set can
	be constructed as an index pair $(U\setminus U_2,\partial U_2)$, and its Conley index is
	$\bar{0}$. According to the conclusion (i) above, we know that the attractor $A_1$ and
	invariant set $S$ are contained in $U\setminus U_2$. Thus by Wedge Sum Property~\ref{th-2} listed
	in Appendix~\ref{sec-appendix}, there must exist connecting orbit from $S$ to attractor $A_1$.
	Otherwise, their Conley index will not $\bar{0}$ but $\Sigma^1 \vee \Sigma^0$, which is a
	contradiction. Similarly, there must be connecting orbit from $S$ to the attractor $A_2$.
	So the proof is completed.
\end{proof}

\begin{figure}[tbhp]
	\centering
	\includegraphics[width=0.5\textwidth]{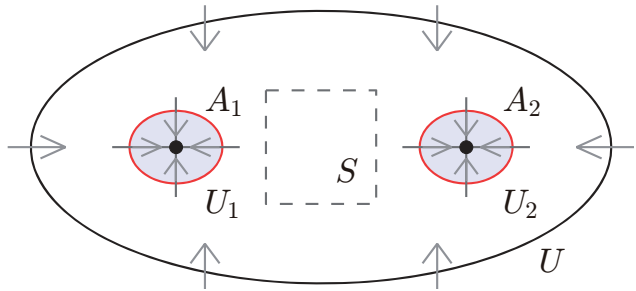}
	\caption{Diagram of the bistable structure. Only $A_1$ and $A_2$ are the two
		asymptotically stable equilibria in the interior of $U$. $U_1$ and $U_2$, which are disjoint
		in $\mathrm{Int}(U)$, are the open neighborhoods of $A_1$ and $A_2$, respectively.}
	\label{figure-1}
\end{figure}

As mentioned in the introduction, more additional information would be necessary if one
wants to get more details about an invariant set. So some of the prerequisites or limitations of
the following study are based on this idea.

To begin with, if it is a one-dimensional system and all equilibria are hyperbolic, i.e., the
derivative of function $f$ evaluated at each equilibrium is not zero, then the invariant set
$S$ is an asymptotically unstable equilibrium. In this case, the bistable structure is relatively
simple, and it must be an unstable equilibrium between two stable equilibria.

Then, if (\ref{eq-2-1}) is a two-dimensional system, we make the following assumption:

\begin{description}
	\item[\textbf{H2}] The dynamical system is structurally stable.
\end{description}

This assumption is mainly based on the fact that most biological objects, such as human
metabolism or species population, are insensitive to small disturbances in the environment
unless some massive changes are encountered. Corresponding to the mathematical model,
it means that the dynamical behavior of the system is unaffected by small perturbations.
Namely, the system is structurally stable, as we assumed above.

Before giving detailed results, we present the following lemma.

\begin{lemma}\label{lemma-2-1-finite}
	For the $n$-dimensional system (\ref{eq-2-1}), there are at most a finite number of hyperbolic
	equilibria within a bounded and closed region $D\subset \mathbb{R}^n$ in the domain of the system.
\end{lemma}

\begin{proof}
	We prove this lemma by contradiction. We first assume that there are an infinite number of
	hyperbolic equilibria within the bounded and closed region $D$. By Bolzano-Weierstrass Theorem
	(sometimes also called Sequential
	compactness theorem) we know that the subset consisting of all possible hyperbolic equilibria,
	denoted by $E$, is sequentially compact due to the compactness of $D$. That is, for each
	sequence of points in $E$ it has a convergent subsequence converging to a point in $E$. For example,
	for sequence $\{e_n\}_{n=1}^{\infty}\subset E$, its subsequence $\{e_{n_k}\}_{k=1}^{\infty}$ converge
	to point $e^*\in E$. In other words, for all $\varepsilon>0$ there exists an $N\in \mathbb{N}^+$
	such that when $k>N$ we have $|e_{n_k}-e^*|<\varepsilon$.
	
	On the other hand, by Hartman-Grobman Theorem, we know that there exists a small neighborhood $U^*$
	of $e^*$, such that the solutions to system (\ref{eq-2-1}) are homeomorphic to that
	of linearization of (\ref{eq-2-1}) at $e^*$ as long as it is inside $U^*$. System (\ref{eq-2-1})
	have the unique equilibrium $e^*$ inside $U^*$, due to the uniqueness of the equilibrium to linearization.
	That is, there exists $\varepsilon_1>0$, such that for any other equilibrium $e$ if it exists we have
	$|e-e^*|>\varepsilon_1$. As a result, this contradicts the statement above, so there are at most a finite
	number of hyperbolic equilibria within $D\subset\mathbb{R}^n$. The proof is completed.
\end{proof}

Then we introduce the concept of `loop' to simplify the description of the following content.
By `loop' we mean the invariant set that consists of equilibrium and complete orbits and
is homeomorphic to one-sphere, $S^1$. Some examples are a periodic orbit, and an invariant
set that is consist of a homoclinic orbit and the corresponding saddle, and invariant set that
is composed of heteroclinic orbits and related equilibria.

\begin{theorem}\label{th-2-2}
	If the system (\ref{eq-2-1}) is two-dimensional and \textbf{H1-H2} are satisfied, then
	\par (i) There must be a finite and odd number of equilibria in $S$, namely
	$2k+1, k\in \mathbb{N}$ equilibria, in which $k+1$ equilibria are the saddles, and $k$ equilibria
	are unstable nodes or focuses;
	\par (ii) There is at most one loop in $S$, which is an unstable limit cycle or a loop consisting
	of saddles, unstable nodes or focuses, and the heteroclinic orbits that flow from
	the latter to the former. Besides, if the loop does exist, it contains an attractor
	inside it, and another attractor is outside of it.
\end{theorem}

\begin{proof}
	(i) We prove that the invariant set $S$ contains a finite number of equilibria first. By \textbf{H1}, we
	know that region $U$ is bounded and closed. According to \textbf{H2} and the
	Andronov-Pontryagin criterion~\cite{Andronov1937, Kuznetsov2013}, we obtain all the equilibria
	in $U$ are hyperbolic. Then by Lemma~\ref{lemma-2-1-finite}, we can conclude that there are
	at most a finite number of hyperbolic equilibria in $U$. So the number of equilibria in $S$ is
	also finite.
	
	Then, we prove that $S$ contains an odd number of equilibria. The winding numbers both
	of the boundary $\partial U$ and each attractor are +1. Then by the fact that the winding number
	of a closed curve is equal to the sum of the winding numbers of all isolated equilibria contained
	in it, we know that $S$ must contain equilibria other than the two attractors, and the sum of
	the winding number of these equilibria is -1. Also, all equilibria are hyperbolic due to \textbf{H2}.
	So we can divide these equilibria into two types according to their winding number: saddle,
	whose index is -1, and non-saddle, whose index is +1. Node and focus belong to the
	non-saddle class. Then $S$ must contain an odd number of equilibria, where the amount of
	the saddle is one more than the number of non-saddle. Otherwise, the sum of the winding
	number of total equilibria in $S$ can not be -1. So in mathematical terms, $S$ contains $2k+1$
	equilibria, where $k+1$ ones are saddles, and $k$ are non-saddle. Besides, because of
	\textbf{H1}, the stable equilibrium cannot be included in $S$. Thus, the non-saddle can only
	be unstable nodes or focuses.
	
	(ii) We prove it by the following four steps:
	
	The first step is to show that the possible loop is an unstable limit cycle or a loop consisting
	of saddles, unstable nodes or focuses, and the heteroclinic orbits that flow from
	the latter to the former. If there is no equilibrium on the loop, it will be an
	unstable limit cycle. That is because $S$ does not contain other attractors, including cycle
	attractor, by \textbf{H1} and the system is also structurally stable by \textbf{H2}. If the loop
	contains equilibrium, it will be the loop consisting of saddles, unstable nodes or focuses, and
	the heteroclinic orbits that flow from the latter to the former. Because the asymptotically
	stable equilibria cannot appear again on the loop by \textbf{H1}, and both the homoclinic
	orbits and the heteroclinic orbits that flow from one saddle to another saddle are not present
	due to their structural instability by \textbf{H2}. So the loop can only be one of these two
	candidates if any.
	
	The second step is to give the position of the loop if it exists, and we will show that the loop
	must be around an attractor. In other words, two attractors are located inside and outside the
	loop, respectively. If it is not the case, then the loop either does not contain any attractor or
	contains two attractors.
	
	For the first case, we assume that the loop does not contain any attractor. So there must be
	an infinite number of orbits in any small neighborhood inside the loop, and they flow further
	into the interior of the loop. Then their $\omega-$limit sets are either cycles or points. The
	cycles are impossible because they must be hyperbolic according to structural stability.
	Stable cycles cannot appear, and unstable cycles cannot be served as $\omega-$limit sets.
	Besides, it is also impossible for their $\omega-$limit sets to be points. Because an unstable
	node or focus cannot be served as the limit set as well, and the stable equilibrium is not
	contained inside this loop. The saddle is also impossible because the amount of saddles is
	finite by conclusion (i), and they cannot be the $\omega-$limit sets of an infinite number of
	orbits. Therefore, the assumption is not correct, and the loop must contain at least one attractor.
	
	For the second case, we assume that there are two attractors inside the loop. Then there
	must be an infinite number of orbits in the small neighborhood outside the loop, which are
	further away from this loop but still stay inside the region $U$. Their $\omega-$limit sets are
	either cycles or points, but both are impossible based on the same analysis as the first case.
	So the loop cannot contain two attractors either.
	
	In short, neither case is possible. So if the loop does exist, it must precisely contain one
	attractor.
	
	The third step is to prove that there is at most one loop around each attractor. Assume there
	are two or more loops around each attractor. Then one can find two adjacent loops, say
	$C_1$, $C_2$, and $C_1$ is inside $C_2$. There must be an infinite number of orbits in the
	small neighborhoods outside of $C_1$ and inside of $C_2$. These orbits flow away from the
	two loops, respectively, and remain between these two loops. Then contradictions can be
	obtained based on the same analysis as the second step once again. Therefore, there is at most
	one loop around each attractor.

	The final step is to show that it is impossible to have a loop for each attractor. We assume
	that is not true, then the outer neighborhood of each loop has an infinity number of orbits,
	which respectively away from these two loops and are still stay in the region $U$. Here the
	contradiction occurs again. So the assumption is not valid.
	
	From the four steps above, we conclude that $S$ contains at most one loop, which is an
	unstable limit cycle or a loop composed of saddles, unstable nodes or focuses, and
	heteroclinic orbits. Moreover, if the loop does exist, it must contain an attractor.
	The proof is completed.
\end{proof}

In the context above, the Conley index of the invariant set $S$ and its possible components were
studied. Now, we are going to consider the configuration of the orbits globally outside of the
invariant set $S$.

Here according to whether the invariant set $S$ contains the loop structure, we make the
following two mutually exclusive assumptions:

\begin{description}
	\item[\textbf{H3}] Invariant set $S$ does not contain loop structure;
	\item[\textbf{H3$^\prime$}] Invariant set $S$ contains loop structure;
\end{description}

Furthermore, we also consider two additional cases according to the connectivity of $S$:

\begin{description}
	\item[\textbf{H4}] The invariant set $S$ is connected;
	\item[\textbf{H4$^\prime$}] The invariant set $S$ is disconnected, in which it has a finite number of
	connected components.
\end{description}

Note that it is beyond the capability of the Conley index to judge whether the invariant set $S$ is
connected or not, unless with the help of other information. Here, we are trying to study all
possible configurations to understand clearly the internal structure of region $U$.

In the case of connected, we can obtain the following lemma.

\begin{lemma}\label{le-2-1}
	If the system (\ref{eq-2-1}) is two-dimensional and \textbf{H1-H4} are satisfied, then for
	invariant set $S$, there must be precisely two disjoint orbits, such that the points on it will
	gradually approach $S$ as time goes to infinity. However, on the other orbits nearing $S$
	points will be gradually away from $S$ finally.
\end{lemma}

\begin{proof}
	We prove this lemma by contradiction. We assume that there exist $n$ orbits, on which points
	gradually approach the invariant set $S$, but $n$ is not equal to 2. Then these $n$ orbits must
	be located on the stable manifolds of the saddles, and the orbits nearing these $n$ orbits will
	away from the invariant set $S$ finally. From \textbf{H4}, we know that $S$ is connected, and from
	\textbf{H2-H3}, we know that $S$ does not have the loop structure, which can be thought of as a simple
	closed curve in a plane. Thus $S$ is simply connected,
	and it is homotopic to a point, denoted by $P_s$. Then $P_s$ have $n$ stable manifolds, which are
	separated by $n$ unstable manifold. So the Conley index of this point is the wedge sum of $n-1$
	multiple of $\Sigma^1$. Since the Conley index is algebraic topological invariant, the Conley index
	of original invariant set $S$ is the wedge sum of $n-1$ multiple of $\Sigma^1$ as well.  This result
	contradicts the conclusion (i) in Theorem~\ref{th-2-1}, so there must be precisely two orbits
	considered above. The proof is completed.
\end{proof}

\begin{theorem}\label{th-2-3}
	If system (\ref{eq-2-1}) is two-dimensional and \textbf{H1-H4} are satisfied, then there must
	exist two different points, $p_1$ and $p_2$, on $\partial U$, such that the orbits through $p_1$
	or $p_2$ will approach invariant set $S$, that is, their $\omega-$limit sets are subset to $S$.
	In addition, the orbits passing through other points on $\partial U$ will approach the
	corresponding attractor $A_1$ or $A_2$.
\end{theorem}

\begin{proof}
	By Lemma~\ref{le-2-1}, we obtain that near $S$ there precisely exist two different orbits, on
	which point are gradually approach $S$, that is to say, their $\omega-$sets are subset to $S$.
	We denote
	these orbits by $\varphi_1$ and $\varphi_2$, respectively. On the boundary of $U$, $\partial U$,
	there must exist two distinct points $p_1$ and $p_2$, such that $\varphi_1$ and $\varphi_2$ flow
	into $U$ through $p_1$ and $p_2$, respectively. Otherwise, we assume that $\varphi_i, i=1,2$,
	comes from the interior of $U$. Thus $\varphi_i$ either flow from $A_i$ or flow from $S$, which
	are both impossible.
	
	We denote the $\omega-$sets of $\varphi_i$ by $E_i$, $i=1,2$. Thus $E_i\in S$. Since $S$ is
	connected, there must exist a path from $E_1$ to $E_2$. Thus the two orbits, $\varphi_1$ and
	$\varphi_2$, together with invariant set $S$, can divide the region $U$ into two sub-regions, as shown
	in Fig.~\ref{figure-2}. For each sub-region, from \textbf{H1} and Lemma~\ref{le-2-1}, we obtain
	that it is a positively invariant set. So its limit set is nonempty. And then from \textbf{H3},
	we know that there are no periodic orbits. So by Poincar\'e-Bendixson Theorem, the limit set must
	be equilibrium, namely $A_1$ or $A_2$. In this case, the two stable equilibrium must be evenly
	distributed into two sub-regions. The proof is completed.
\end{proof}

\begin{figure}[tbhp]
	\centering
	\includegraphics[width=0.5\textwidth]{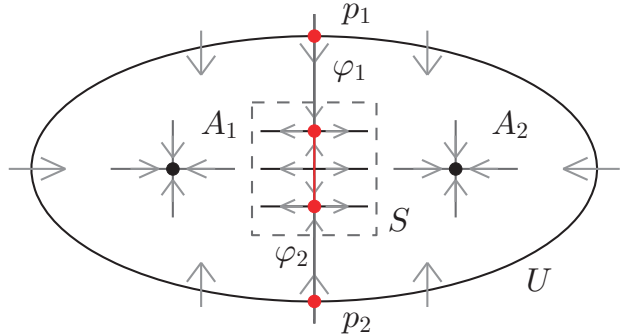}
	\caption{Diagram of the invariant set $S$ in the bistable structure. This is a kind of possible
		configuration of invariant set $S$, in which there are two saddles and an unstable node, and
		$S$ is connected.}
	\label{figure-2}
\end{figure}

The Theorem above describes the destination of the orbits through the boundary $\partial U$,
i.e., their $\omega-$limit sets. Here, we call the connected set that is composed of orbits
$\varphi_1$, $\varphi_2$, and invariant set $S$ as separatrix. This is due to the fact that
inside of region $U$, the orbits at each side of the connected set will approach different
attractors, respectively.

Next, we consider the disconnected case. For this, we assume that $S$ has $K$ connected
components and label these components by $S_1, S_2,\cdots,S_K$. Similar to
Lemma~\ref{le-2-1} and Theorem~\ref{th-2-3}, we have the following results.


\begin{lemma}\label{le-2-2}
	If system (\ref{eq-2-1}) is two-dimensional and \textbf{H2-H3} are satisfied, and the connected
	invariant set $S^*$ is composed of saddle, unstable nodes or focus and connecting orbits, in
	addition, its Conley index is $\bar{0}$, then there must be precisely one orbit, such that
	the points on it will gradually approach $S^*$ as time goes to infinity. However, on the other
	orbits neighboring $S^*$, points will run away from $S^*$ finally.
\end{lemma}

\begin{proof}
	
	The proof is similar to the one for Lemma~\ref{le-2-1}, and we prove it by
	contradiction as well. Since $S^*$ is connected and from \textbf{H2-H3}, we know that it
	does not have a loop structure and $S^*$ is simply connected. So $S^*$ can
	be homotopically shrunk to a point, denoted by $P_{s^*}$. We assume that there exist
	precisely $n$ orbits, on which point will gradually approach $S^*$ as time goes to infinity,
	and $n$ is not equal to 1. Then the Conley index of $P_{s^*}$ is the wedge sum of the
	$n-1$ multiple of $\Sigma^1$. Since the Conley index is an algebraic topological invariant, the
	Conley index of the original invariant set $S^*$ is also the wedge sum of the $n-1$ multiple of
	$\Sigma^1$. This contradicts $H(S^*)=\bar{0}$, so there must be one orbit as required. Thus
	the points on other orbits will be away from $S^*$ finally. The proof is completed.
\end{proof}

\begin{theorem}\label{th-2-4}
	If system (\ref{eq-2-1}) is two-dimensional and \textbf{H1,H2,H3,H4$^\prime$} are
	satisfied, then there must be $K+1$ disjoint points, $p_1, \cdots, p_K, p_{K+1}$, on $\partial U$,
	such that the orbits through $p_i$ will approach invariant set $S$. That is to say, their
	$\omega-$limit set is a subset of $S$. In addition, orbits passing through other points
	on $\partial U$ will approach corresponding attractor, $A_1$ or $A_2$.
\end{theorem}

\begin{proof}
	From \textbf{H4$^\prime$} and Property~\ref{th-2} listed in Appendix~\ref{sec-appendix}, one has
	\begin{equation}
		H(S) = H(S_1) \vee H(S_2) \vee \cdots \vee H(S_K).
	\end{equation}
	Due to the conclusion that $H(S) = \Sigma^1$ in Theorem~\ref{th-2-1}, for invariant set $S$, there
	must be one connected component whose Conley index is $\Sigma^1$ and all the Conley index of other
	$K-1$ connected component is the additive identity~\cite{Conley1978}, namely $\bar{0}$.
	Otherwise, the Conley index of $S$ will be the Wedge sum of some items, which contradicts
	$H(S) = \Sigma^1$.
	
	To make it easier to describe, we re-label $K$ connected components by exchanging the
	label of the first component and the component whose Conley index is $\Sigma^1$. In the process,
	if the Conley index of component $S_1$ is exactly $\Sigma^1$, then the label method remains the
	same. After this process, the Conley index of the first component $S_1$ will be $\Sigma^1$, and
	the  Conley index of others will be $\bar{0}$.
	
	For connected component $S_1$, by Lemma~\ref{le-2-1} above, we know that there must be two orbits
	whose $\omega-$limiet sets are subset to $S_1$, and we label them by $\varphi_1^1, \varphi_1^2$.
	Similar to the Theorem~\ref{th-2-3}, $\varphi_1^1$, $\varphi_1^2$ must come from outside of
	region $U$. We denote the intersection points of these two orbits and boundary $\partial U$
	by $p_1^1$, $p_1^2$.
	
	For each of other components, say $S_i, i=2,\cdots,K$. By Lemma~\ref{le-2-2} above we also know
	that there must be one and only one orbit $\varphi_i$ whose limit set is subset to $S_i$. This orbit must come from
	outside of $U$ as well, and we denote the corresponding intersection point by $p_i$.
	
	Finally, we get $K+1$ different points on the boundary $\partial U$, and the orbits through
	them all approach the corresponding connected component of invariant set $S$. For other points on
	the boundary, the orbits that through them finally approach one of the two attractors.
\end{proof}

Similar to the connected case, there is also a separatrix in the disconnected case. The difference
is that this separatrix is composed of the connected component $S_1$, which Conley index is
$\Sigma^1$, and the two orbits, $\varphi_1^1$ and $\varphi_1^2$. On the boundary $\partial U$,
the points on the same side of this separatrix will approach the same attractor except a finite
number of points, namely $p_i, i=2,\cdots,K$ in the proof above.

To summarize, the separatrix acts as a role of the threshold, and the point on the boundary
can be considered as an initial condition. Except for a finite set of initial boundary points that
tend to the invariant set $S$, almost all the initial points tend to one of the two attractors.
More importantly, those initial points located on the different sides of separatrix will have different
destinies.


\begin{theorem}\label{th-2-5-new}
	If system (\ref{eq-2-1}) is two-dimensional, and \textbf{H1,H2,H3$^\prime$,H4} are satisfied,
	then there must be one and only one point on the boundary of $U$, denoted by $p_1$,
	such that its $\omega-$limit set is a subset of $S$, and within the region bounded by $S$
	some attractor, say $A_1$, is contained. Besides, all the other points on the boundary flow
	to another attractor, $A_2$.
\end{theorem}

\begin{proof}
	First, by Theorem~\ref{th-2-2}(ii) and \textbf{H3$^\prime$}, there is one and only one loop in $S$, which is
	an unstable limit cycle or a loop consisting of saddles, unstable nodes or focuses, and
	heteroclinic orbits flowing from the latter to the former. Besides, this loop must contain
	an attractor, say $A_1$, and another attractor $A_2$ is outside this loop, as shown in Fig.~\ref{figure-2-1}.

	Next, we can get that $S$ and the region bounded by $S$, where there are attractor $A_1$ and
	heteroclinic orbits flowing to $A_1$, form a simply connected domain. Its Conley index is
	$\bar{0}$ because its isolating neighborhood can be obtained by digging out a small open
	neighborhood of $A_2$, like $U_2$ in Theorem~\ref{th-2-1}, and the exit set is precisely the
	boundary $\partial U_2$.
	
	Next, we prove that there is only one orbit that flows to $S$ from the outside of $S$. There is
	only one loop in $S$, as shown in Theorem~\ref{th-2-2} above, so that we can deform
	homotopically the region bounded by this loop into a point. Then this point is asymptotically
	unstable since the original loop is unstable. As a result, the region bounded by $S$ becomes
	a new invariant set, denoted by $S^\prime$, which is consists of unstable nodes or focuses,
	saddles, and heteroclinic orbits, and its Conley index is still  $\bar{0}$. Therefore, it can be
	obtained from Lemma~\ref{le-2-2} that there is only one orbit that tends to $S^\prime$. So there
	only exists an orbit that flows to invariant $S$ as well, and the $\omega-$limit set of this
	orbit must be a saddle, which is a subset of $S$ but does not on the loop in $S$.

	The orbit approaching $S$ must flow from the outside of $U$, and it must intersect the
	boundary $\partial U$ at a point, which is $p_1$. The other orbits passing through
	$\partial U$ will only eventually flow towards $A_2$. The proof is completed.
\end{proof}

\begin{figure}[tbhp]
	\centering
	\includegraphics[width=0.5\textwidth]{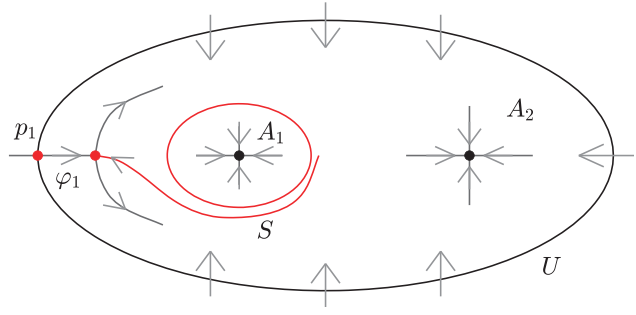}
	\caption{Diagram of the invariant set $S$ in the bistable structure. This is a kind of possible
		configuration of invariant set $S$, in which there are a saddle and an unstable limit cycle, and
		$S$ is connected.}
	\label{figure-2-1}
\end{figure}

\begin{theorem}\label{th-2-6-new}
	If the system (\ref{eq-2-1}) is two-dimensional and \textbf{H1,H2,H3$^\prime$,H4$^\prime$} are satisfied,
	then on the boundary of $U$, there must be $K$ different points, whose $\omega-$limit set is a subset
	of $S$, and other points eventually flow to attractor $A_2$.
\end{theorem}

\begin{proof}
	It can be obtained from \textbf{H3$^\prime$} and Theorem~\ref{th-2-2} that there is only one loop
	structure in $S$, and this loop bounds the attractor $A_1$. Then, the invariant set
	$S$ and the region it surrounds, containing $A_1$ and related heteroclinic orbits flowing to it,
	constitute a new invariant set, denoted by $S^\prime$. It also has $K$ connected components,
	and each component is simply connected. The Conley index of this new invariant set is the
	addition identity, $\bar{0}$, and the computation method is the same as that of Theorem~\ref{th-2-5-new}.
	Then, by Property~\ref{th-2}, the Conley index of each component is $\bar{0}$, as well.
	
	For the new variant set $S^\prime$, the structure of the connected component containing the loop
	is the same as that of the combined structure of $S$ with the region bounded by it in Theorem~\ref{th-2-5-new}.
	Therefore, for this component, there must be only one orbit tending to it, denoted by $\varphi_1$.
	For the other $K-1$ components, according to Lemma~\ref{le-2-2}, there is only one orbit tending
	to each component, denoted respectively by $\varphi_i, i =2,3,\cdots,K$. Here $K$ orbits must come from
	the outside of $U$, and they intersect with the boundary $\partial U$ at $K$ different points.
	Therefore, for the $K$ orbits passing through these $K$ points, their $\omega-$limit sets are
	$K$ components respectively subset to $S$, and all of the orbits passing through the
	remaining boundary points will eventually flow to the attractor $A_2$. The proof is completed.	
\end{proof}

When a loop exists, it does have another kind of separatrix, namely, the cycle separatrix.
Different from the separatrix above, cycle separatrix divides the region $U$ into two sub-regions
that are not homotopy equivalent. The first one is an annular region, which is outside of the separatrix,
and another is simply connected, which is inside of the separatrix. The loop structure in Theorem~\ref{th-2-5-new}
and Theorem~\ref{th-2-6-new} just happen to be the cycle separatrix. Inside of it, almost all orbits
flow towards attractor $A_1$, while outside of the loop, almost all orbits, including the orbits
coming from outside of $U$, approach attractor $A_2$.

Finally, for a three or higher dimensional system, it is hard to obtain its global behavior. One
reason is that more complicated structures can occur here, such as strange attractors.
Another reason is that there is no higher-dimensional analog of
the Andronov-Pontryagin criterion\cite{Andronov1937}, which makes it difficult to analyze the
structural stability in high dimensional systems.
However, employing the generalized winding number\cite{Jia2020}, we can still obtain the analogous result in
higher dimensional systems, if all of the equilibria in the system (\ref{eq-2-1}) are hyperbolic.

\begin{description}
	\item[H2$^\prime$] The equilibria in the dynamical system are hyperbolic.
\end{description}

\begin{theorem}\label{th-2-7-new}
	For system (\ref{eq-2-1}), if \textbf{H1,H2$^\prime$} are satisfied, then there must be a finite
	and odd number of equilibria in $S$.
\end{theorem}

\begin{proof}
	The prove is the same as for Theorem~\ref{th-2-2}. We first obtain that there is a finite number of
	hyperbolic equilibria inside of $U$ by assumption \textbf{H1, H2$^\prime$}, and Lemma~\ref{lemma-2-1-finite},
	and so is the number of equilibria in invariant set $S$.
	
	Then we prove that there is an odd number of equilibria within the region $U$. Because all the equilibria
	in the system are
	hyperbolic by \textbf{H2$^\prime$}, its winding numbers are either +1 or -1, which is dependent
	on the dimension of the system and its unstable manifolds. Besides, the winding number of boundary
	$\partial U$ is either +1 or -1 as well. Thus the number of these equilibria must be odd.
	Otherwise, the sum of their winding number will be even by the Additivity Theorem of the
	generalized winding number \cite{Jia2020}. The proof is completed.
\end{proof}

In the preceding Theorems or Lemmas, \textbf{H1} requires orbits must enter the interior of $U$
immediately. This restriction is a little bit harsh, and it can be relaxed slightly as below.
Now the outside orbits are allowed to run along the boundary $\partial U$ for a while
once they touch the boundary, but eventually, they must enter the interior of the $U$ the same as before.
In other words, the boundary $\partial U$ can contain orbit segments that comes from the
outside of $U$ and eventually enters the inside of $U$. So the new assumption is

\begin{description}
	\item[\textbf{H1$^\prime$}] There exists a closed region $U$ in the domain of system (\ref{eq-2-1}),
	such that there are precisely two asymptotically stable equilibria in the interior of $U$,
	say $A_1$ and $A_2$, and on the boundary $\partial U$, orbits run from outside into
	interior immediately, or firstly run along the boundary for a while then into its interior.
\end{description}

With the assumption \textbf{H1} replaced by \textbf{H1$^\prime$}, the invariant set $S$
within $U$ and its connection with two attractors are the same as those in previous studies.
Thus Lemma~\ref{lemma-2-1-finite} to \ref{le-2-2}, and Theorem~\ref{th-2-1},
\ref{th-2-2}, and \ref{th-2-7-new}, are still valid,
but Theorem~\ref{th-2-3} to \ref{th-2-6-new} need to be modified appropriately.
That is because the last four Theorems involve the points in boundary $\partial U$.
And the expression ``disjoint points" in each theorem or lemma should be replaced with
``disjoint points or boundary segments."
Here, we do not take these related results as corollaries or new theorems. Accordingly,
their proofs are omitted as well since they are almost the same as the proofs in related theorems.
No matter whether the boundary is composed of points or orbit segments, as long as they all run
into the interior of $U$, the internal structure of $U$ will not be affected.


\subsection{Case 2: A cycle attractor and a one-point attractor}
\label{sec-2-2}

In this case, we consider the second bistable structure, which contains a cycle attractor and
a one-point attractor. Moreover, we only study it in two-dimensional situations since loop structure,
such as periodic orbits, cannot exist in a one-dimensional system.

For the two-dimensional case, we first give the following two assumptions, which have the
same boundary conditions as \textbf{H1}, and the relative positions of two attractors are
also considered.

\begin{description}
	\item[\textbf{H5}] There exists a closed region $U$ in the domain of system (\ref{eq-2-1}), such that
	there are precisely two attractors, an asymptotically stable periodic orbit $C_1$ and
	an asymptotically stable equilibrium $A_1$, where $A_1$ is inside $C_1$, and on the boundary
	$\partial U$, all orbits run from outside into interior immediately;
	\item[\textbf{H6}] There exists a closed region $U$ in the domain of system (\ref{eq-2-1}), such that
	there are precisely two attractors, an asymptotically stable periodic orbit $C_1$ and
	an asymptotically stable equilibrium $A_1$, where $A_1$ is outside $C_1$, and on the boundary
	$\partial U$, all orbits run from outside into interior immediately;
\end{description}

Similar to Theorem~\ref{th-2-1}, we have the following results.

\begin{theorem}\label{th-2-5}
	If systems (\ref{eq-2-1}) is two-dimensional and \textbf{H5} is satisfied, then
	\par(i) There is other isolated invariant set, denoted by $S$, in the annular region bounded by
	the circle $C_1$ and the point $A_1$, and its Conley index is $\Sigma^2\vee\Sigma^1$.
	\par(ii) There are the connecting orbits from $S$ to both attractors, $C_1$ and $A_1$, as well.
\end{theorem}

\begin{proof}
	The proof is similar to that of Theorem~\ref{th-2-1}.
	\par (i) Firstly, we construct an isolating neighborhood of invariant set $S$. We can
	find an open neighborhood $U_1$ of $A_1$ since $A_1$ is asymptotically stable, such that the orbits
	passing through the boundary $\partial U_1$ will flow into the interior of $U_1$ and be
	attracted by $A_1$ in the end. Besides, we can also find an annular neighborhood $R_1$ of $C_1$,
	which is bounded by two simple closed curves $L_1$ and $L_2$, as shown in Fig.~\ref{fig-17},
	such that the orbits passing through the boundary points will run into the interior of $R_1$ and
	be attracted by attractor $C_1$. Then the annular region bounded by curves $\partial U_1$ and $L_1$
	is an isolating neighborhood, and the exit set is precisely composed of two boundary curves $\partial U_1$ and $L_1$.
	So its Conley index is $\Sigma^2\vee\Sigma^1$ computed by collapsing the exit set into a point. By the Wa\.zewski
	Property~\ref{th-1} of Conley index, we obtain that its isolated invariant set is not empty, and there must
	exist other isolated invariant set $S$ except two attractors.

	(ii) The region bounded by curve $L_1$ is an isolating neighborhood, and its Conley index is $\Sigma^2$,
	which does not equal the wedge sum of the two Conley indices of invariant $A_1$ and $S$,
	$\Sigma^0\vee(\Sigma^2\vee\Sigma^1)$. So by Wedge Sum Property~\ref{th-2} of Conley index, in addition to
	$S$ and $A_1$, there must be at least one connecting orbit from $S$ to $A_1$. Similarly, we
	also obtain the existence of connecting orbit from $S$ to $C_1$.The proof is completed.
\end{proof}

\begin{figure}[tbhp]
	\centering
	\includegraphics[width=0.5\textwidth]{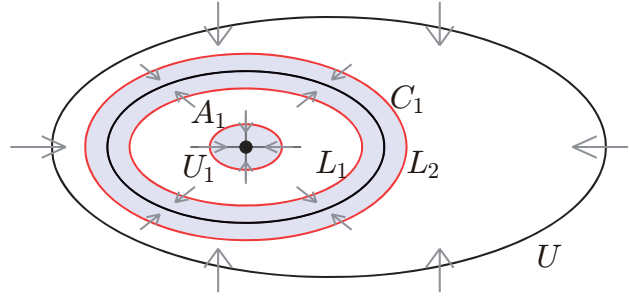}
	\caption{Diagram of bistable structure with a cycle attractor $C_1$ and a one-point attractor $A_1$,
		where $A_1$ is inside of $C_1$.}
	\label{fig-17}
\end{figure}

We should note that the system considered in \textbf{H5} and \textbf{H6} contains the periodic orbits,
which are a kind of loop structure, so we do not consider the assumption \textbf{H3}(or \textbf{H3$^\prime$})
in the current situation. Nevertheless, considering \textbf{H2} is still rational.  Given this, we have the following result.

\begin{theorem}\label{th-2-9-new}
	If the system (\ref{eq-2-1}) is two-dimensional and \textbf{H5} and \textbf{H2} are satisfied,
	then
	\par(i) There must be a finite and even number of equilibria in $S$, namely $2k$, $k=0,1,\cdots$,
	in which $k$ equilibria are saddles, and other $k$ equilibria are unstable nodes or focuses.
	\par(ii) There is one and only one loop structure in $S$, and it contains $A_1$ in its interior.
	\par(iii) Invariant set $S$ is connected.
	\par(iv) Denote the isolated invariant set within the region bounded by cycle attractor
	$C_1$ and boundary $\partial U$ by $S_2$. Then it must contain a finite and even number of
	equilibria, which is the same as (1).
\end{theorem}

\begin{proof}
	(i) The proof of (i) is omitted here since its strategy and process are analogous
	to that of Theorem~\ref{th-2-2}.
	
	(ii) The proof of (ii) is divided into the following four steps:
	\par The first step is to give the possible loop structure in $S$.
	There are two candidate types, an unstable limit cycle and the loop structure consisting of
	saddles, unstable nodes or focuses, and heteroclinic orbits. The loop must be an unstable
	limit cycle if there are no equilibria on it. Moreover, it must be the second type of loop if there
	are some equilibria on it. That is because there is no possibility of an additional stable equilibrium
	in $S$ by \textbf{H5}, and also, no homoclinic orbit of the saddle flow to itself and the
	heteroclinic orbit of saddle flow to another saddle according to the structure stability
	assumption in \textbf{H2}.

	The second step is to illustrate the position of the loop, that is, the loop, if any, must contain
	$A_1$ inside it. We prove it by contradiction. We first assume that the region bounded by the
	loop does not contain $A_1$. Then in any small neighborhood inside either loop mentioned in
	the previous step, there must be an infinite number of orbits, which further flow into the interior
	of the loop. Then the $\omega-$limit sets of these infinite orbits are either equilibria or periodic orbits.
	The periodic orbit is impossible because the stable periodic orbit cannot appear again due to \textbf{H5},
	and the unstable periodic orbit cannot be served as a $\omega-$limit set. Besides, the equilibria
	are also impossible. Because unstable focus or nodes cannot be used as a limit set as well, and
	stable equilibria cannot reappear. Saddles are also impossible because their total number is finite,
	and it is impossible for an infinite number of orbits with a finite number of saddles as their $\omega-$limit
	set. So the assumption we made is not correct, and the loop must contain $A_1$ inside of itself.

	The third step is to prove the existence of the loop by contradiction as well.
	First, we assume that a loop structure of any kind does not exist. Then the invariant set
	$S$ must consist of saddles, unstable nodes or focuses, and heteroclinic orbits according to
	Theorem~\ref{th-2-9-new}(i). Next, freely choose a saddle in $S$, such
	as $s_1$, whose two stable manifolds must come from distinct unstable nodes or focuses, say
	$n_1$ and $n_2$. We continue to check if both equilibria are connected to other saddles.
	It may assume that $n_1$ is not, but $n_2$ is connected to another saddle, denoted by $s_2$.
	Next, we continue to consider the stable manifold of $s_2$. By analogy, we can always end
	this process because the number of equilibria in $S$ is finite. In the end, we can get
	a simply connected invariant set similar to  Fig.~\ref{figure-18}(a), and its Conley index is $\Sigma^2$
	due to its asymptotical instability. We continue to count all other saddles, then get
	a finite number of the invariant set like Fig.~\ref{figure-18}(a). If this process does not pick up all the unstable
	nodes or focuses, they can be considered as unstable invariant set disconnected from each other, and
	their Conley indices are still $\Sigma^2$. Finally, we get that the invariant sets $S$ is a collection
	of invariant set just like Fig.~\ref{figure-18}(a) and unstable equilibria. In this view, the Conley index of
	$S$ is the wedge sum of a finite number of $\Sigma^2$, which contradicts the Theorem~\ref{th-2-5}(i).
	So there must be a loop structure.

	The final step is to prove the uniqueness of the loop. We assume this is not true. That means
	there are two or more loops. From the second step, we know that $A_1$ must be contained in the
	smallest loop, which is then contained in the second smallest loop again. And so forth, just like
	the Russian doll. These loops are finite due to the compactness of region $U$ and
	the structural stability of the system. Next, let us consider the two outermost loops, as shown in
	Fig.~\ref{figure-18}(b) (other cases are similar). Any small inner neighborhood of the outer loop must
	have an infinite number of orbits that further flow into the interior of this loop. Their $\omega-$limit sets
	are either equilibria or loops. However, neither is possible through the same analysis as the second step.
	So the assumption is not valid, and the loop structure is unique.

	\begin{figure}[tbhp]
		\centering
		(a)
		\includegraphics[width=0.4\textwidth]{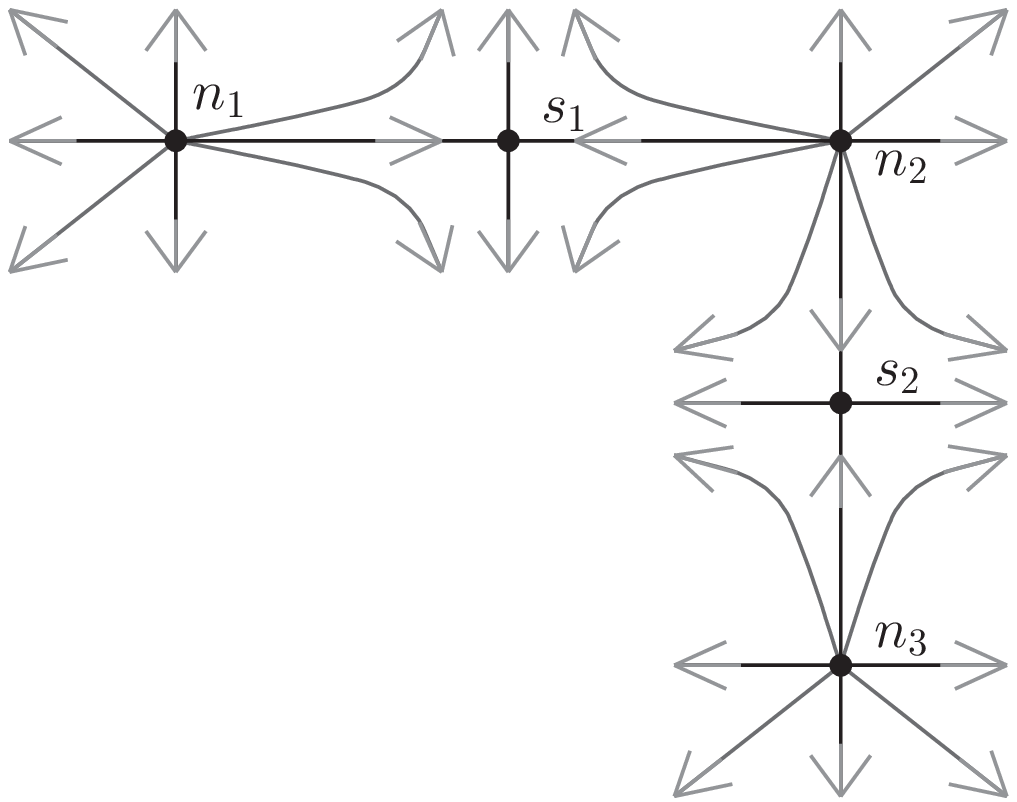}
		\quad
		(b)
		\includegraphics[width=0.45\textwidth]{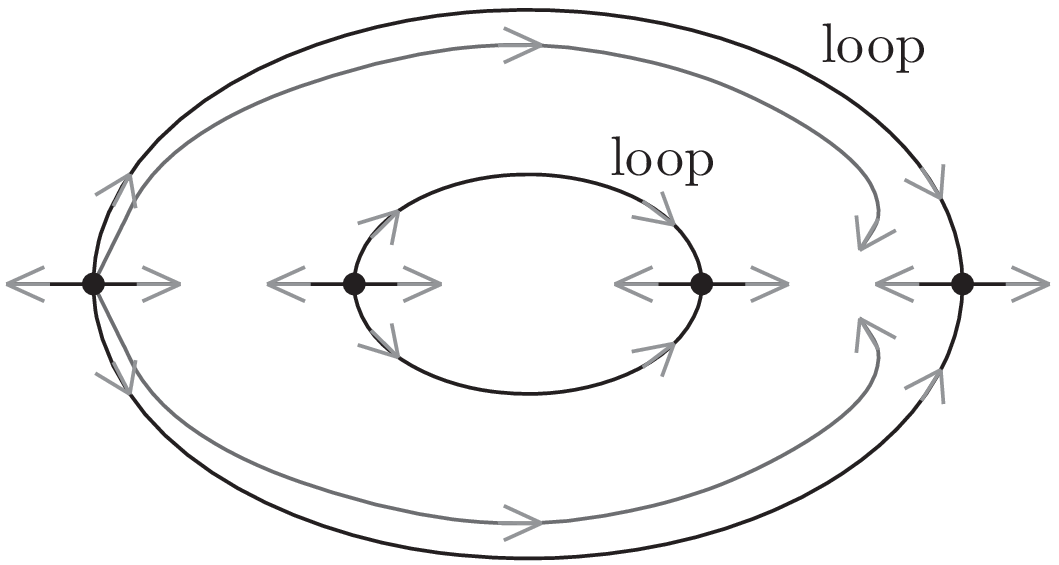}
		\caption{(a)A schematic diagram of one possible unstable invariant set consists of saddles
			and unstable nodes or focuses. (b)Schematic diagram of the loop structure.}
		\label{figure-18}
	\end{figure}
	
	(iii) We assume that the invariant set $S$ is not connected. We know that $S$ contains a loop,
	and the Conley index both of this loop and $S$ is $\Sigma^2\vee\Sigma^1$. Therefore, if there are
	other invariant sets, which are not connecting to the loop and also disconnected from each other,
	then their Conley index must be $\bar{0}$ by Property~\ref{th-2}. From (i), we know that $S$ only
	can contain unstable nodes or focuses, saddles, and heteroclinic orbits. Also, Lemma~\ref{le-2-2}
	shows that there must be an orbit that is approaching this invariant set. Because the invariant
	sets are disconnected from each other, this orbit cannot come from $S$, and it cannot come from
	other invariant sets and $A_1$ as well. So the contradiction occurs, and $S$ must be connected.
	
	(iv) We omit the proof of (iv) again because it can be proved via the winding number and the same idea
	as Theorem~\ref{th-2-2}. The proof is completed.
\end{proof}

Similar to the case of two one-point attractors with the loop structure appears, the loop in
Theorem~\ref{th-2-9-new} is the cycle separatrix. Almost all the orbits outside and inside of the separatrix flow
towards the cycle attractor $C_1$ and the one-point attractor $A_1$, respectively.

\begin{theorem}\label{th-2-10}
	If the system (\ref{eq-2-1}) is two-dimensional and \textbf{H6} is satisfied, then
	\par(i) There is an isolated invariant set, denoted by $S_1$, in the region bounded by cycle
	attractor $C_1$, and its Conley index is $\Sigma^2$.
	\par(ii) There is also an isolated invariant set, denoted by $S_2$, outside the cycle
	attractor $C_1$ but except the one-point attractor $A_1$, and its Conley index is $\Sigma^1$.
	Besides, there are connecting orbits from $S_2$ to $C_1$ and $A_1$, respectively.
\end{theorem}

\begin{proof}
	(i) Known from \textbf{H6}, $C_1$ is an asymptotically stable periodic orbit. Therefore,
	there must be a simple closed curve, denoted by $L_1$, in the small inner neighborhood of $C_1$,
	so that the orbits passing through $L_1$ will flow to its outside
	and eventually approach $C_1$, as shown in Fig.~\ref{figure-5-new}.
	Thus, the closed region bounded by $L_1$ together with
	the curve $L_1$ form an index pair, and its Conley index is $\Sigma^2$, which is not
	equal to $\bar{0}$. So, according to Property~\ref{th-1}, there must be an isolated
	invariant set, $S_1$, in the region bounded by $L_1$, and it is also within the region bounded
	by $C_1$.
	
	\begin{figure}[tbhp]
		\centering
		\includegraphics[width=0.5\textwidth]{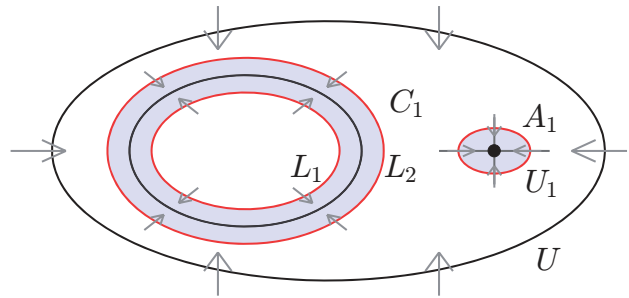}
		\caption{Diagram of bistable structure with a cycle attractor $C_1$ and a one-point attractor $A_1$,
			where $A_1$ is outside of $C_1$.}
		\label{figure-5-new}
	\end{figure}
	
	(ii) Similar to the $L_1$ above, we can also find a simple closed curve $L_2$ in the small
	outer neighborhood of $C_1$ so that the orbits passing through $L_2$ will flow to $C_1$.
	Besides, for $A_1$, we can find an open neighborhood $U_1$ as well, similar to Theorem~\ref{th-2-1},
	so that the orbits passing through the boundary will flow to $A_1$. Next, we dig out the
	interior of $L_2$ along with $U_1$, resulting in a region with two holes. As a result, this
	region, together with the boundary of these two holes, constitutes an index pair, whose
	Conley index is not $\bar{0}$ but $\Sigma^1$. So there is also an isolated invariant set
	$S_2$ between $C_1$ and $A_1$.
	
	Finally, we can prove the existence of connecting orbits by the same idea as
	Theorem~\ref{th-2-1}. Firstly, digging out the region bounded by $L_2$ forms the index pair
	$(U\setminus \mathrm{Int}(L_2), L_2)$, which proves the existence of connecting orbit from $S_2$
	to $A_1$.  Similarly, digging out the region $U_1$ shows the presence of connecting orbit
	from $S_2$ to $C_1$. The proof is completed.
\end{proof}

When the system is two-dimensional, and \textbf{H6, H2} are satisfied, we can refer to the
situation of two one-point attractors. That is because if we see the cycle attractor from its outside,
the cycle attractor and the one-point attractor have the same dynamics. Moreover, if we contract
homotopically the region bounded by the cycle attractor, it would be a one-point attractor. In this case,
there can be different separatrices. When the unstable loop structure appears, which must
around $C_1$ or $A_1$, this loop structure would be the cycle separatrix. However, when
the unstable loop structure does not appear, the separatrix would be similar to the situations
discussed below of Theorem~\ref{th-2-3} and Theorem~\ref{th-2-4}.

Corresponding to the previous subsection, the assumption \textbf{H5} and \textbf{H6} can also
be relaxed to the followings, respectively.

\begin{description}
	\item[\textbf{H5$^\prime$}] There exists a closed region $U$ in the domain of system (\ref{eq-2-1}), such that
	there are precisely two attractors, an asymptotically stable periodic orbit $C_1$ and
	an asymptotically stable equilibrium $A_1$, where $A_1$ is inside $C_1$, and on the boundary
	$\partial U$, all orbits run from outside into interior either immediately or firstly run along the
	boundary for a while then into its interior.
	\item[\textbf{H6$^\prime$}] There exists a closed region $U$ in the domain of system (\ref{eq-2-1}), such that
	there are precisely two attractors, an asymptotically stable periodic orbit $C_1$ and
	an asymptotically stable equilibrium $A_1$, where $A_1$ is outside $C_1$, and on the boundary
	$\partial U$, all orbits run from outside into interior either immediately or firstly run along the
	boundary for a while then into its interior.
\end{description}

Here, we do not consider the coexistence of a cycle attractor and a one-point attractor in higher-dimensional
systems. The first reason is that the dynamic behavior around the periodic orbit in a high-dimensional
system is complicated. The second reason is that there can be strange attractors in high-dimensional
systems, which lead to more combinations of bistables.

\section{Examples and applications}
\label{sec-3}

In this section, some examples are provided to indicate that
the bistable structures we considered are common phenomena in
certain biological systems. At the same time, some applications are given
to illustrate how to use the conclusions of the previous section.

\subsection{ Some examples in biological systems}

As mentioned in the introduction, the bistable structure can occur in many other
biological models, such as the insect outbreak model \cite{Ludwig1978}, genetic control model \cite{Griffith1971},
interactive model between two species~\cite{Yan2013,Yue2013,Tainaka2006,Tainaka2016},
epidemic model~\cite{Wang2006,Zhang2008,Zhou2012} and so on. Here, we briefly introduce
four biological models, all of which display the bistable structure we considered.

The first one is the spruce budworm population model that was proposed and analyzed by Ludwig et al.
\cite{Ludwig1978,Ludwig1997,Strogatz2018}. This model is aimed at studying the outbreak of spruce
budworm in eastern Canada. This kind of insect is a serious pest, for it can defoliate and kill most
of the fir trees in the forest in about four years.

This model can be written as follows:

\begin{equation}\label{eq-3-7}
	\frac{dN}{dt} = RN\left( 1-\frac{N}{K} \right) - \frac{BN^2}{A^2+N^2}.
\end{equation}
Where $N(t)$ denotes the budworm population at time $t$. In this model, $N(t)$ is assumed to have a
logistic growth rate, as well as a saturated mortality rate due to predation, chiefly by birds.
By analysis, we find that model (\ref{eq-3-7}) can have four equilibria, the origin $o$, and
the other three positive equilibria, as shown in Fig.~\ref{figure-6}. There are precisely two
stable equilibria, $a$ and $c$, which correspond to the refuge level and outbreak level of
budworm, respectively, and can be seen as two point-attractors. It is effortless to find a closed
interval covering $a$ and $c$ to
the right of origin as the region $U$ in \textbf{H1}. In this case, equilibrium $b$ acts as
a separatrix, since orbits to the left side of $b$ flow to attractor $a$, while orbits to
another side flow to $c$.

\begin{figure}[tbhp]
	\centering
	\includegraphics[width=0.5\textwidth]{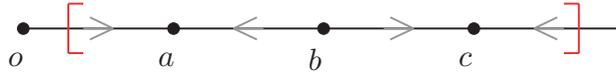}
	\caption{The bistable structure diagram for the model (\ref{eq-3-7}).}
	\label{figure-6}
\end{figure}

The second one is the genetic control model discussed by Griffith~\cite{Griffith1971,Strogatz2018}.
We all know that in a cell, genes are transcribed into message RNA (mRNA), and then mRNA acts as
a template to produce proteins. Here, this model assumes that the activity of the genes considered
is directly induced by a certain amount of protein for which it codes. In other words, genes control
the production of proteins, and in turn, proteins stimulate the genes. In this way, an autocatalytic
feedback process is formed. The model can be written in the following dimensionless form:

\begin{equation}\label{eq-3-8}
	\left\{
	\begin{aligned}
		\frac{dx}{dt} &= -ax + y,\\
		\frac{dy}{dt} &= \frac{x^2}{1+x^2} -by .\\
	\end{aligned}
	\right.
\end{equation}
Here variables $x(t)$ and $y(t)$ denote the concentration of mRNA and protein at time $t$,
respectively. Either rate of growth is proportional to the other, and both will be degraded, which
are controlled by parameters $a$ and $b$, respectively. We should note that
there is no equation for DNA since it is externally supplied in this process.

\begin{figure}[tbhp]
	\centering
	(a)
	\includegraphics[width=0.45\textwidth]{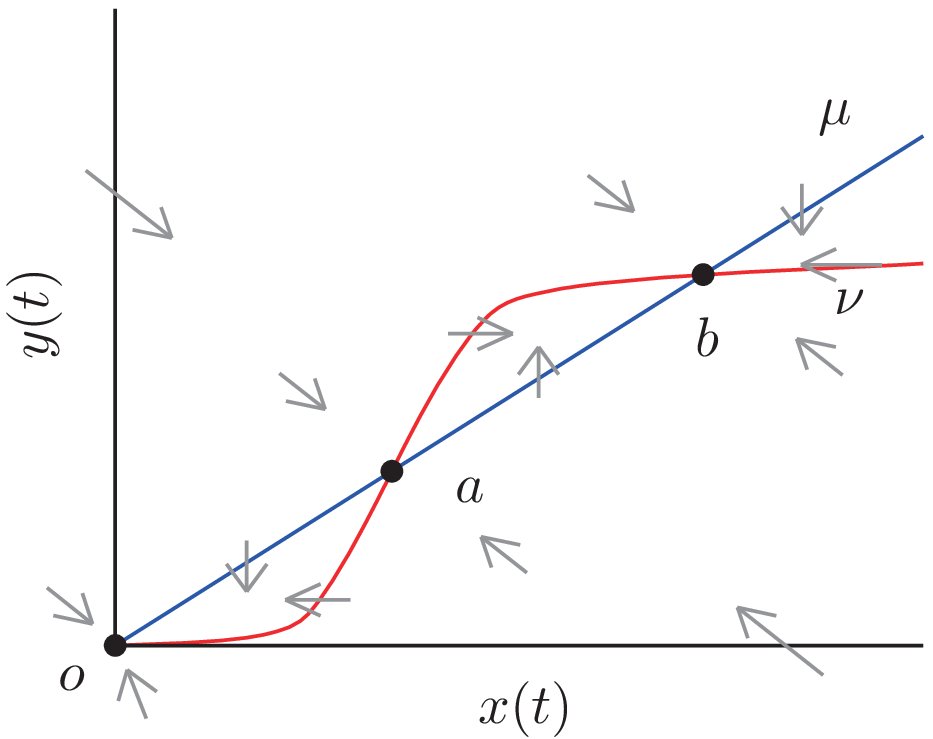}
	(b)
	\includegraphics[width=0.45\textwidth]{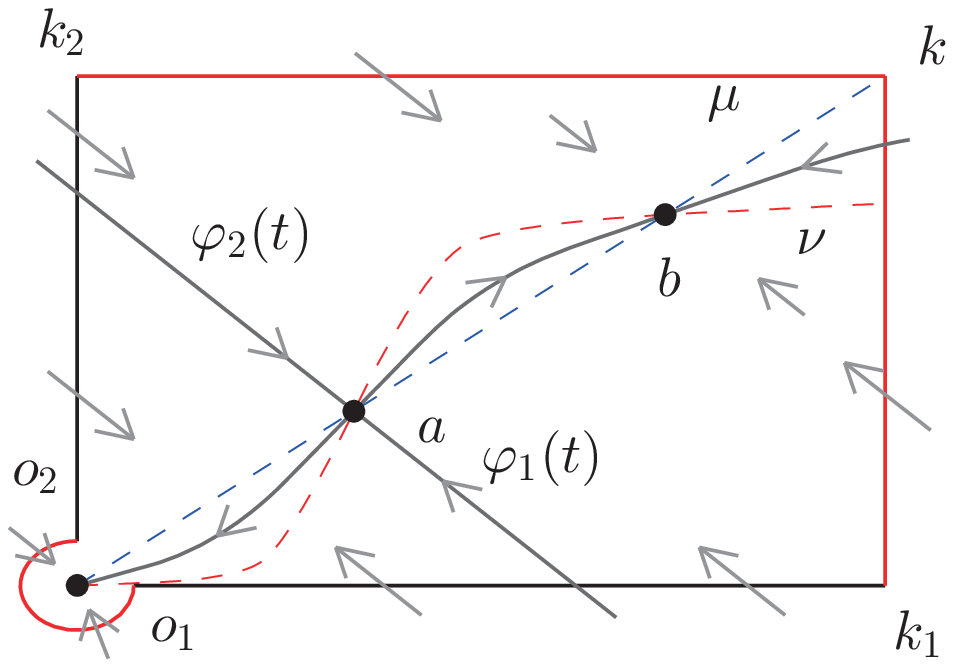}
	\caption{The bistable structure diagram for the model (\ref{eq-3-8}).
		(a)Line $\mu$ and Sigmoidal curve $\nu$ are $x$-nullcline and $y$-nullcline, respectively.
		Moreover, there are three equilibria here, $o$, $a$, and $b$, where $o$ and $b$ are stable.
		(b)The region $U$ is bounded by four line segments and one arc. Furthermore, the separatrix is
		composed of two orbits, $\varphi_1(t)$ and $\varphi_2(t)$, and equilibrium $a$.}
	\label{figure-7}
\end{figure}

For model (\ref{eq-3-8}), we can easily get its $x$-nullcline, $y=ax$, and $y$-nullcline,
$y=\frac{x^2}{b(1+x^2)}$. They are denoted by $\mu$ and $\nu$, respectively, as before.
Under certain conditions, line $\mu$ and Sigmoidal curve $\nu$ can intersect transversely, thus
obtaining three equilibria $o$, $a$, and $b$, which are schematically shown in Fig.~\ref{figure-7}(a).
By linearization, we obtain $o$ and $b$ are two stable equilibria, and $a$ is unstable.
Then we can apply the conclusions in the previous section by constructing a region $U$
that contains these two stable equilibria precisely. As a result, equilibrium $a$ must be a saddle.
This result also can be obtained by linearization. However, the purpose of doing
this here is to illustrate that this model shows the bistable structure we are interested in.

For this model, the separatrix is composed of two orbits, $\varphi_1(t)$ and $\varphi_2(t)$, and
one saddle, $a$. To the bottom left of the separatrix, the orbits flow to origin $o$, which
corresponds to the silent state of the gene. However, on the other side of the separatrix, the
orbits flow to equilibrium $b$, which corresponds to the activated state. Its biological interpretation
is that genes can be switched between two opposite states, silent and activated, depending on
the external environment. In other words, the concentrations of mRNA and the corresponding
protein determine if the gene is expressed.

The third model is a predator-prey system with group defense, which is studied by Freedman, Wolkowicz
\cite{Freedman1986}, and other researchers \cite{Xiao2001,Zhu2003}. Group defense describes
the phenomenon that predation gets harder, even prevented altogether, due to the growing ability
of the prey to defend or disguise themselves when their population is large enough, so it is considered
as a nonmonotonic functional response of predator to prey density. The study of
this model was about to exhibit the so-called paradox of enrichment, and, interestingly, the second case of
the bistable structure is observed here as well, namely, the bistable structure of the coexistence of a
cycle attractor and a one-point attractor.

The model is written as
\begin{equation}\label{eq-3-9-new}
	\left\{
	\begin{aligned}
		\frac{dx}{dt} &= xg(x,K)-yp(x),\\
		\frac{dy}{dt} &= y(-s+q(x)) .\\
	\end{aligned}
	\right.
\end{equation}
Here $x(t)$ and $y(t)$ denote the densities of prey and predator at time $t$, respectively.
Also, functions $g(x,K)$, $p(x)$, and $q(x)$ are assumed to be continuously differentiable and
are represent the specific growth rate of the prey in the absence of predation, predator
response function, and the conversion rate of prey to predator, respectively. For their requirements,
please refer to the literature \cite{Freedman1986}.  Here, we only consider the model with the following forms,
and it is also researched in \cite{Freedman1986}.


\begin{equation}\label{eq-3-10-new}
	\left\{
	\begin{aligned}
		\frac{dx}{dt} &= 2x(1-\frac{x}{K})-\frac{9xy}{x^2+3.35x+13.5},\\
		\frac{dy}{dt} &= y(-1+\frac{11.3x}{x^2+3.35x+13.5}).\\
	\end{aligned}
	\right.
\end{equation}

With the increasement of parameter $K$, the carrying capacity of prey, the model
exhibits different dynamics. When $K=4$, equilibrium $e_\lambda$ is  asymptotically stable,
and its position is illustrated in  Fig.~\ref{figure-10}(a).
If $K$ increases to $K=6$ due to enrichment, $e_\lambda$ becomes unstable, and a unique
asymptotically stable periodic orbit around it will appear. Then when $K$ is growing to $K=7$,
the periodic orbit will coalesce with a homoclinic orbit that is stable from within and unstable from
without. Here, we focus on the situation of $K=6$, as shown in Fig.~\ref{figure-10}.
By linearization, $e_K$ is an asymptotically stable node, $e_\mu$ is a saddle, and $e_\lambda$ is asymptotically
unstable, as shown in Fig.~\ref{figure-10}(b). Besides, by numerical analysis, there is also an asymptotically
stable periodic orbit around the $e_\lambda$, neighboring which orbits will approach it finally,
as shown in Fig.~\ref{figure-10}(c). In this case, the system exactly embeds the second bistable structure,
and the asymptotically stable periodic orbit is the cycle attractor $C_1$ and the asymptotically stable
equilibrium $e_K$ is the one-point attractor $A_1$.

\begin{figure}[tbhp]
	\centering
	(a)
	\includegraphics[width=0.45\textwidth]{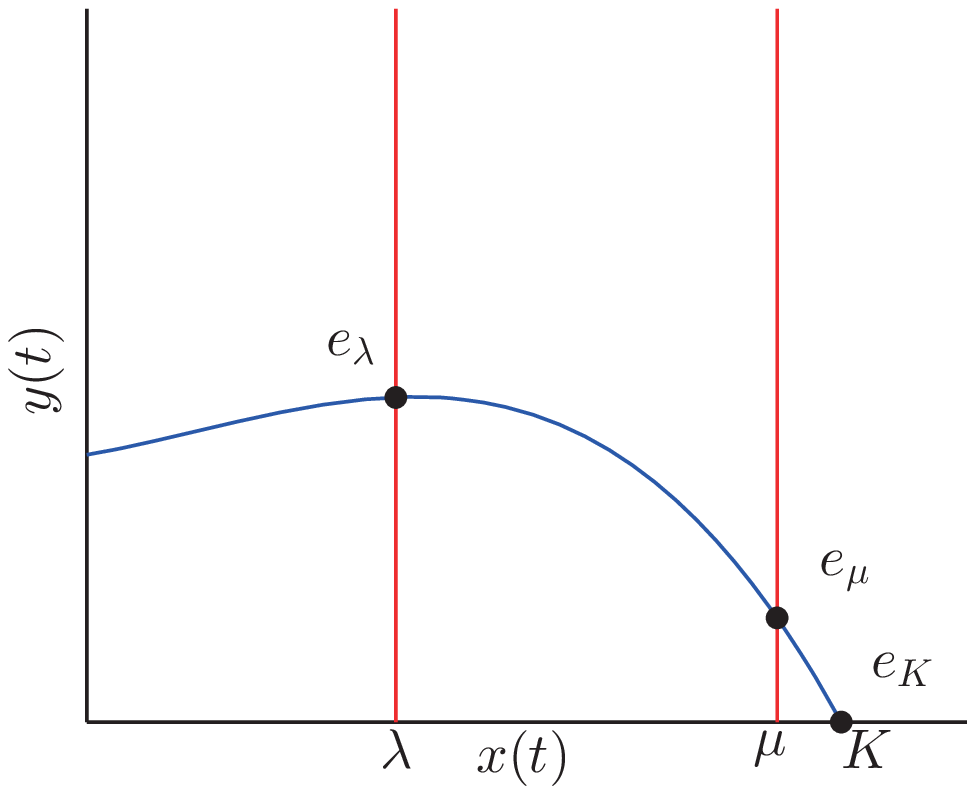}
	(b)
	\includegraphics[width=0.45\textwidth]{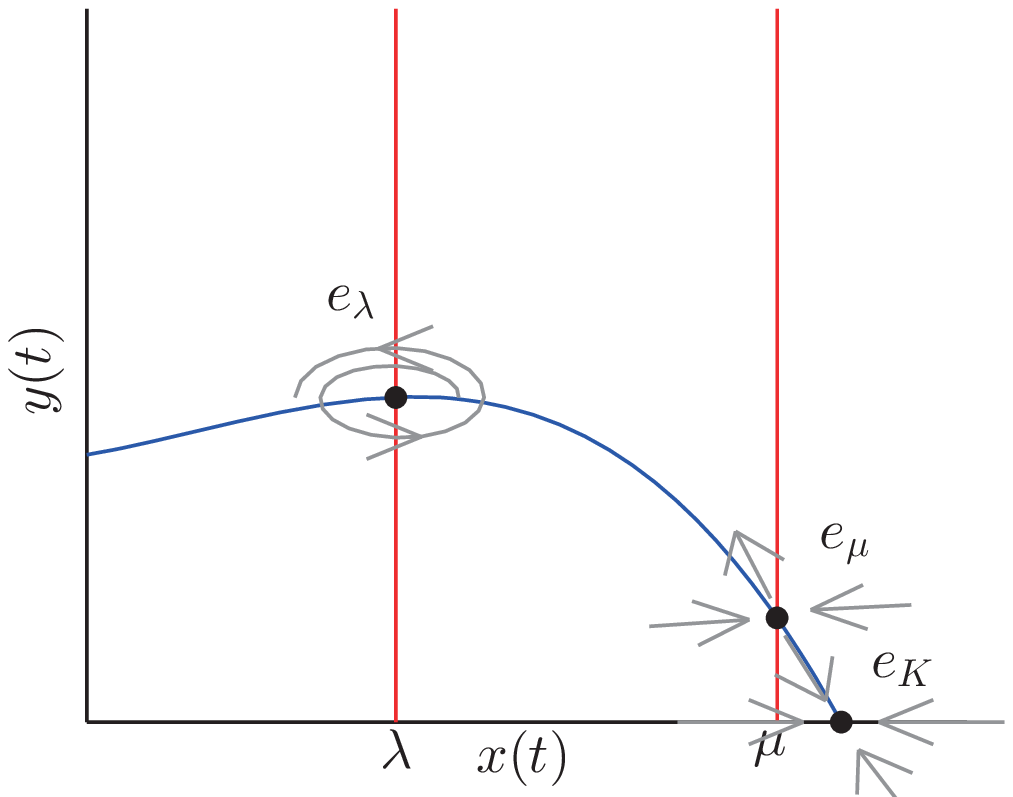}\\
	(c)
	\includegraphics[width=0.45\textwidth]{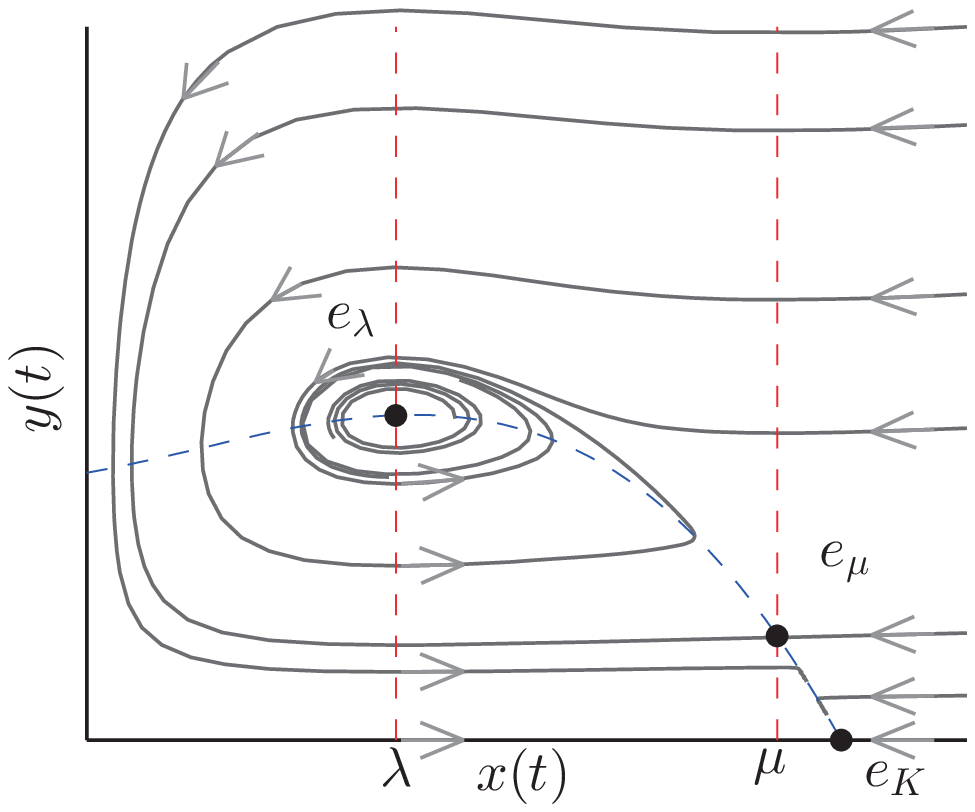}
	(d)
	\includegraphics[width=0.45\textwidth]{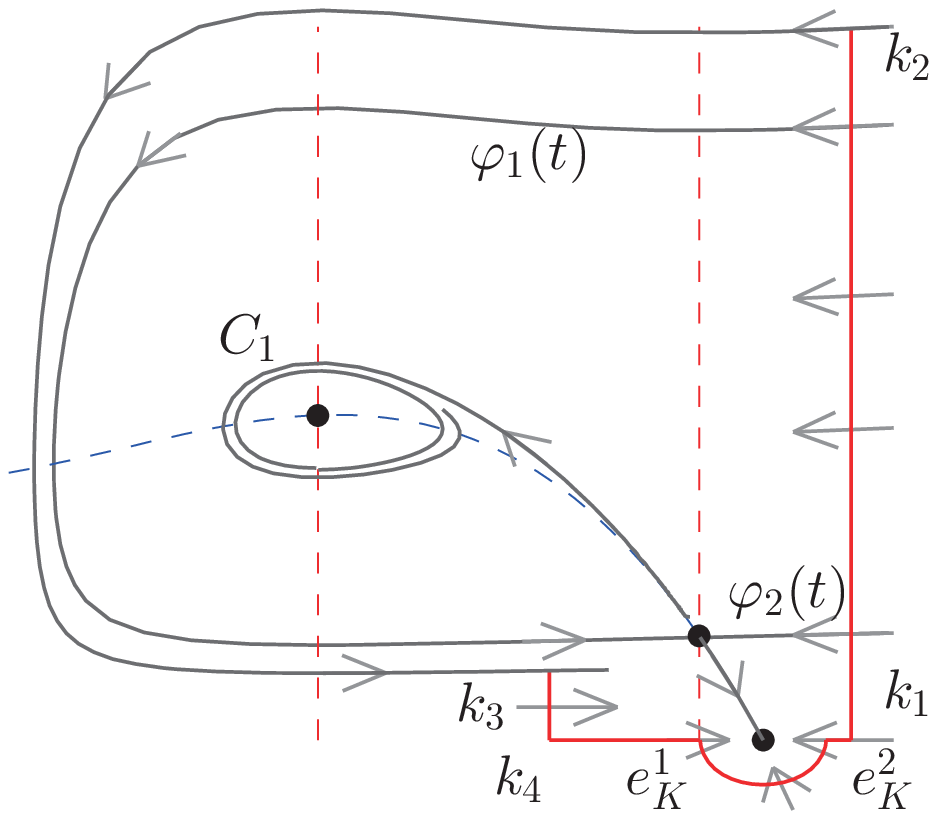}
	\caption{The bistable structure diagram for model \ref{eq-3-10-new}. Where, $K=6$.
		(a)The $y$-axis and the curve are $x$-nullclines, and $x$-axis,  $x=\lambda$, and $x=\mu$
		are $y$-nullclines.
		(b)Equilibrium $e_\lambda$ is asymptotically unstable, Equilibrium $e_\mu$ is a saddle, and
		Equilibrium $e_K$ is asymptotically stable.
		(c)Numerical solutions with different initial values are drawn.
		(d)The region $U$ required is constructed. Orbits $\varphi_1(t)$, $\varphi_2(t)$, and
		equilibrium $e_\mu$ form the separatrix.}
	\label{figure-10}
\end{figure}

The region $U$ required in \textbf{H6$^\prime$} can be constructed as illustrated in  Fig.~\ref{figure-10}(d),
which is bounded sequentially by line segment $\overline{k_1k_2}$, orbit segment $\widetilde{k_2k_3}$,
line segments $\overline{k_3k_4}$, $\overline{k_4e_K^1}$, arc $\widehat{e_K^1e_K^2}$, and line segment
$\overline{e_K^2k_1}$. The asymptotically stable periodic orbit $C_1$ and stable equilibrium $e_K$ are
the only two attractors within region $U$. As described in (\ref{th-2-10}), within the region bounded
by $C_1$, there is an unstable equilibrium $e_\lambda$, which is an isolated invariant set, and outside
periodic orbit $C_1$, there is a saddle $e_\mu$. Besides, there also exist connecting orbits from $e_\mu$
to both periodic orbit $C_1$ and equilibrium $e_K$. Besides, the combination of orbits $\varphi_1(t)$,
$\varphi_2(t)$, and saddle $e_\mu$ are served as the separatrix of region $U$, on each side of which orbits
approach different attractors. Its biological meaning is that different initial states will lead to different
destinies, coexistence, or extinction of predators.


The final one is an SIR compartment model in epidemiology, which was analyzed by Zhou and Fan~\cite{Zhou2012}.
Here the limited medical resources and their supply efficiency were considered, which was inspired by
the treatment functions in~\cite{Wang2006} by Wang and in~\cite{Zhang2008} by Zhang.
Backward bifurcation, a kind of transcritical bifurcation, may occur, and one can also find
bistability here. By convention, the total population is divided into three compartments:
susceptible (S), infected (I), and removed (R). The susceptible can become infected by
touching with the infected individuals, and can then recover naturally or receive treatment,
thus becoming a remover. This model is written as

\begin{equation}\label{eq-3-9}
	\left\{
	\begin{aligned}
		\frac{dS}{dt} &= \Lambda - f(S,I) - dS, \\
		\frac{dI}{dt} &= f(S,I) - ( d + \gamma + \varepsilon )I - h(I), \\
		\frac{dR}{dt} &= \gamma I + h(I) -dR. \\
	\end{aligned}
	\right.
\end{equation}
Here $S(t)$, $I(t)$ and $R(t)$ denote the number of people in the three compartments at time $t$,
respectively. $\Lambda$ denotes the recruitment of the population. $d$ and $\varepsilon$ denote the natural
death rate and disease-related mortality, respectively. $\gamma$ denotes the natural recovery rate.
Saturated type incidence function, $f(S,I)=\frac{\beta SI}{1+\kappa I}$, is considered here since it
can reflect the ``psychological" or inhibition effect~\cite{Capasso1978}. Moreover, the treatment function
$h(I)$ is taking the form of $\frac{\alpha I}{\omega + I}$, which is based on the fact of limited medical
resources and supply efficiency.

The change rates of variables $S$ and $I$ are independent of $R$, so we can just focus on the first two
equations in (\ref{eq-3-9}), leading to

\begin{equation}\label{eq-3-10}
	\left\{
	\begin{aligned}
		\frac{dS}{dt} &= \Lambda - \frac{\beta SI}{1+\kappa I} - dS, \\
		\frac{dI}{dt} &= \frac{\beta SI}{1+\kappa I} - ( d + \gamma + \varepsilon )I - \frac{\alpha I}{\omega + I}. \\
	\end{aligned}
	\right.
\end{equation}

For this model, backward bifurcation can occur when parameters $\omega$ and $\alpha$ are located in
a particular region. For details, one can refer to work~\cite{Zhou2012}. In this situation, there are three
equilibria: one disease-free equilibrium, $e_0$, and two endemic equilibria, $e_1$ and $e_2$, as
shown in  Fig.~\ref{figure-8}(a). By linearization, it can be obtained that $e_0$ and $e_2$ are asymptotically stable,
and $e_1$ is a saddle, which is unstable. Alternatively, by the conclusions in the previous section, one can
also get $e_1$ is a saddle. That is because we can construct a required region $U$ which precisely contains
$e_0$ and $e_2$.

The region $U$ is constructed as shown in Fig.~\ref{figure-8}(b), where the line segments
$\overline{oe_0^1}$ and $\overline{ok}$ are located on
the $x$-axis and $y$-axis, respectively. The major arc $\widehat{e_0^1e_0^2}$ is a segment of the boundary of
a certain attracting neighbor of $e_0$. Line segment $\overline{ke_0^2}$ is located on the line
$S+I=\frac{\Lambda}{d}$, and this is because the derivatives $\frac{d(S+I)}{dt}$ at the points on this line
segment are negative, which means the points on this line segment will go into the interior of $U$
as the time goes on. $e_0$ and $e_2$ are the only two attractors in $U$. Furthermore, the existence of
limit cycles can be excluded~\cite{Zhou2012}, and appropriate parameters can be found to make the model
structurally stable. So by Theorem~\ref{th-2-2} and
the number of equilibria, it can be concluded that $e_1$ must be a saddle. In addition, by Theorem~\ref{th-2-3},
orbits $\varphi_1(t)$ and $\varphi_2(t)$ are located on the two stable manifolds of $e_1$.
They, together with equilibrium $e_1$, constitute the separatrix, as sketched in Fig.~\ref{figure-8}(b).
The orbits above it will approach the endemic $e_2$ as time goes on, whereas
the orbits below it will approach the disease-free equilibrium $e_0$.

\begin{figure}[tbhp]
	\centering
	(a)
	\includegraphics[width=0.45\textwidth]{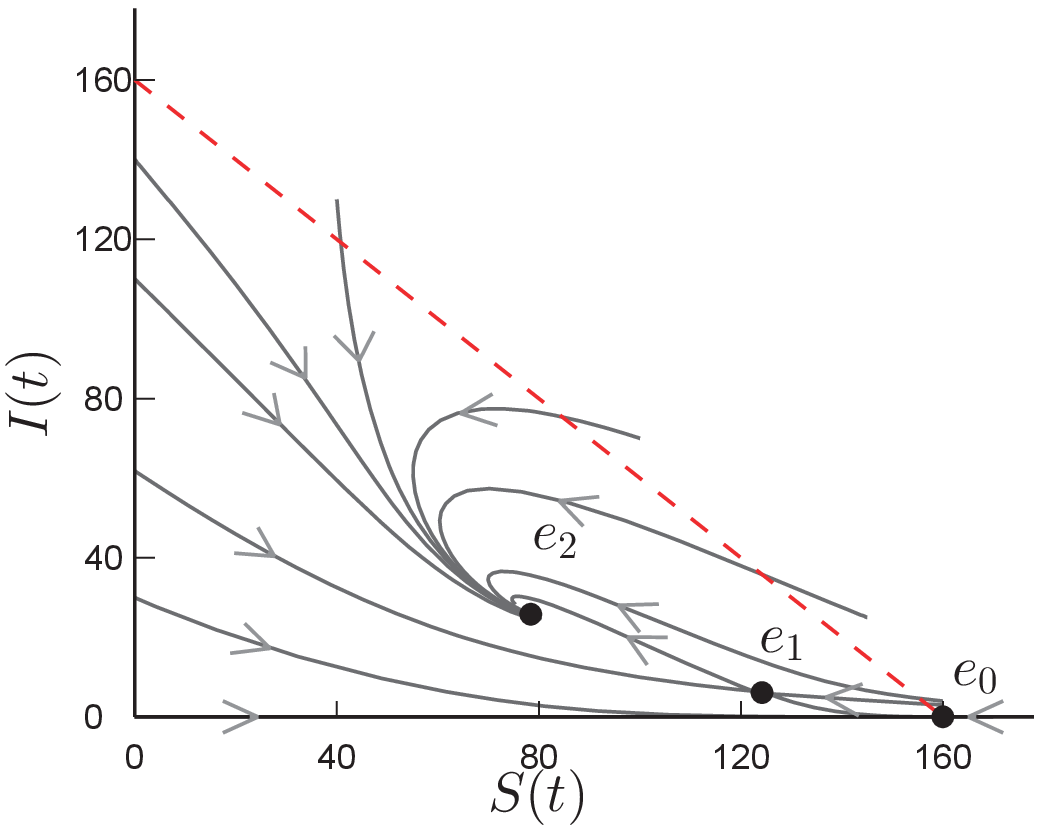}
	(b)
	\includegraphics[width=0.45\textwidth]{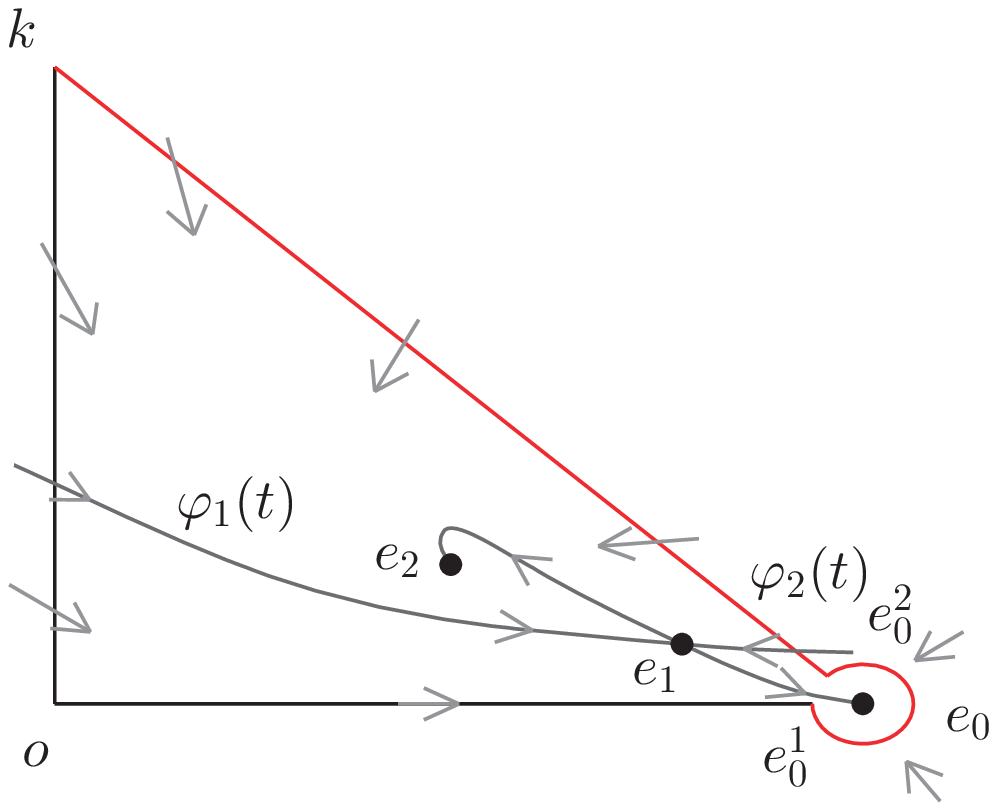}
	\caption{The bistable structure diagram for the model (\ref{eq-3-10}).
		(a)This is the phase portrait of (\ref{eq-3-10}). There are three equilibria,
		$e_0$, $e_1$, and $e_2$, in which $e_0$ and $e_2$ are stable and $e_1$ is a saddle.
		Here $\Lambda=16$, $\beta=0.005$, $\kappa=0.01$, $d=0.1$, $\gamma=0.01$, $\varepsilon=0.02$,
		$\alpha=6$, and $\omega=7$.
		(b)The region $U$ is bounded by two line segments, one orbit section and an arc.
		And the separatrix is composed of two orbits, $\varphi_1$ and $\varphi_2$, and saddle $e_1$.    }
	\label{figure-8}
\end{figure}

\subsection{Applications to a competitive eco-system}
\label{sec-3-1}

Here we consider a competitive model with two species competing for the same food supply
\cite{Murray2002,Strogatz2018,Hirsch2012}, which can exhibit bistability. Instead of analyzing
specific equations, we adopt a different approach; that is, we only consider the qualitative
features to obtain the behavior of the model. This method can give more general conclusions than
analyzing a specific model. At the same time, the difficulty of solving specific equations can be
avoided, which can occur when other detailed factors, such as predators, seasonal effects,
and other sources of food, are taken into account.

First, let $x(t)$ and $y(t)$ denote the populations of the two species at time $t$. For instance,
rabbits and sheep, who are all fed on grass. Thus the model can be written as
\begin{equation}\label{eq-3-1}
	\left\{
	\begin{aligned}
		\frac{dx}{dt} &= M(x,y)x,\\
		\frac{dy}{dt} &= N(x,y)y. \\
	\end{aligned}
	\right.
\end{equation}
Here $M(x,y)$ and $N(x,y)$ are the growth rates of the two species, and both are functions of variable
$x$ and $y$. We consider the following assumptions:

\begin{enumerate}
	\item In the absence of either species, the other would grow to its carrying capacity, $a$ for
	species $x$ and $b$ for $y$. Therefore,
	\begin{equation}
		\begin{array}{c}
			M(x,0)>0 \text{\quad for \quad} x<a \text{\quad and \quad} M(x,0)<0 \text{\quad for \quad} x>a,\\
			N(0,y)>0 \text{\quad for \quad} y<b \text{\quad and \quad} N(0,y)<0 \text{\quad for \quad} y>b.
		\end{array}
	\end{equation}
	\item Because of the interspecies competition, if the population of any species increases, then
	the population of another species will decrease. Thus
	\begin{equation}
		\frac{\partial M}{\partial y} < 0 \text{\quad and \quad} \frac{\partial N}{\partial x} < 0.
	\end{equation}
	\item In addition, if either population is very large, then both the population will decrease. So
	there exists $K>0$ such that
	\begin{equation}
		M(x,y) < 0 \text{\quad and \quad} N(x,y) <0 \text{\quad if \quad} x\geq K \text{\,or\,}y\geq K.
	\end{equation}
\end{enumerate}

\begin{figure}[tbhp]
	\centering
	\includegraphics[width=0.5\textwidth]{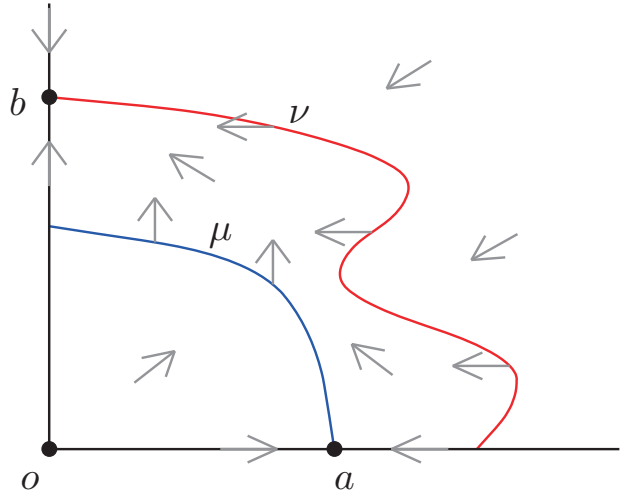}
	\caption{Equilibria and nullclines $\mu$ and $\nu$ for a competitive model.
		Curve $\mu$ corresponds to equation $M(x,y)=0$, and Curve $\nu$ corresponds
		to equation $N(x,y)=0$.}
	\label{figure-3}
\end{figure}

We can see from the model (\ref{eq-3-1}) that no matter what $M(x,y), N(x,y)$ are there always exist
three equilibria: $(0,0)$, $(a,0)$, and $(0,b)$, as shown in Fig.~\ref{figure-3}.
In order to obtain their stability, we compute
the Jacobian:
\begin{equation}
	J \triangleq \left(
	\begin{array}{cc}
		\frac{\partial }{\partial x}(\frac{dx}{dt}) & \frac{\partial }{\partial y}(\frac{dx}{dt}) \\
		\frac{\partial }{\partial x}(\frac{dy}{dt}) & \frac{\partial }{\partial y}(\frac{dy}{dt})
	\end{array}
	\right)
	= \left(
	\begin{array}{cc}
		xM_x + M & xM_y \\
		yN_x & yN_y+N
	\end{array}
	\right),
\end{equation}
where $M_x=\frac{\partial M}{\partial x}$, $M_y=\frac{\partial M}{\partial y}$, and so on.
Then we analyze these equilibria in turn.

$(0,0)$: It is an unstable node. This is because the eigenvalues of matrix $J$ valued at this point are
$M$ and $N$, and both are positive according to assumption 1 above.

$(a,0)$: When it is to the left of $\nu$, it is a saddle. This is because the eigenvalues of matrix $J$
valued at this point are a negative number, $aM_x(a,0)<0$, and a positive number $N(a,0)>0$. However,
when $(a,0)$ is to the right of $\nu$, eigenvalue $N(a,0)$ will be negative. So it is a stable node
in the latter case. In either case, the positive half-axis of $x$ is located on the
stable manifold of $(a,0)$.

$(0,b)$: Similar to $(a,0)$, when $(0,b)$ is below $\mu$, it is a saddle. Conversely, when $(0,b)$
is above $\mu$, it is a stable node. And the positive half-axis of $y$ is located on the stable manifold
of $(0,b)$ correspondingly.

Besides, there can be other equilibrium, say $e$, in the first quadrant, that is,
the intersection point of $\mu$ and $\nu$ if they intersect. One can find that if both
have the negative slope at equilibrium $e$ but $\mu$ is steeper, then it is asymptotically
stable. By implicit differentiation, this condition can also be formulated as
\begin{equation}\label{eq-3-6}
	\text{slope\, of\, } \mu = -\frac{M_x}{M_y} < \text{\,slope\, of\,}\nu = - \frac{N_x}{N_y} <0.
\end{equation}
This can be proved by checking the trace is negative and the determinant is positive for matrix $J$
evaluated at $e$, which means matrix $J_e$ has two negative
eigenvalues. This condition is necessary, as well. Namely, if the intersection point is
asymptotically stable, the equation (\ref{eq-3-6}) must be held at that point.

Next, we begin to study the competitive model (\ref{eq-3-1}) using the results above and the conclusion
in the previous section. Firstly, let the right-hand side of (\ref{eq-3-1}) equal to $0$. Then we get
the $x$-nullclines, $x=0$ and curve $M(x,y)=0$ labeled by $\mu$, and the $y$-nullclines, $y=0$ and
curve $N(x,y)=0$ labeled by $\nu$. We assume that the nullclines can be sketched as (a) in
Fig.~\ref{figure-4}. Based on the above analysis, we obtain that the equilibrium $(0,0)$ labeled $o$
is an unstable node, and equilibrium $(a,0)$ labeled $a$ is a stable node since it is located to the
right of $\nu$. Similarly, the equilibrium $(0,b)$ labeled $b$ is also a stable node. For equilibrium
$c$, the slope of $\nu$ is positive, as can be seen in (b) in  Fig.~\ref{figure-4}, while the slope of
$\mu$ is negative. Thus equilibrium $c$ is unstable.

So far, we have got precisely two stable equilibria, $a$ and $b$. In order to apply the relevant
conclusions in the previous section, we need a region $U$ in which the stable equilibria is exactly
$a$ and $b$. So let us get started dealing with its construction first. Because $a$ is a stable node,
we can find a closed curve surrounding $a$ in its basin of attraction,
such that all the points on this curve will immediately run
into its interior and approach equilibrium $a$ eventually.
Moreover, the two intersections with the
$x-$axis are labeled with $a_1$ and $a_2$, as shown in Fig.~\ref{figure-4}(c). The same
method is applied to equilibrium $b$, and we can get another closed curve and intersections,
$b_1$ and $b_2$. However, a closed curve surrounding $o$ can be found, such that the
internal points, not $o$, will run outside through this curve. The intersections
with the $x-$axis and $y-$axis are labeled with $o_1$ and $o_2$.
By Assumption 3, we know that the direction vector point to the bottom left when the
component $x$ or $y$ is greater than $K$. Then we can find line segments, $\overline{k_1k}$
and $\overline{k_2k}$, to the right of $x = K$ and to the top of $y=K$, respectively, such
that the points on those segments flow to the lower left. The relevant labels
can be seen in  Fig.~\ref{figure-4}(c). Finally, an region bounded by line segments,
$\overline{k_2k}$ and $\overline{k_1k}$, flow sections, $\widetilde{o_1a_1}$,
$\widetilde{k_1a_2}$, $\widetilde{o_2b_1}$, and $\widetilde{k_2b_2}$, and arcs, $\widehat{b_1b_2}$,
$\widehat{0_10_2}$, and $\widehat{a_1a_2}$, are constructed. As shown in Fig.~\ref{figure-4}(d),
this region meets the requirements in \textbf{H1$^\prime$}.

\begin{figure}[tbhp]
	\centering
	(a)
	\includegraphics[width=0.45\textwidth]{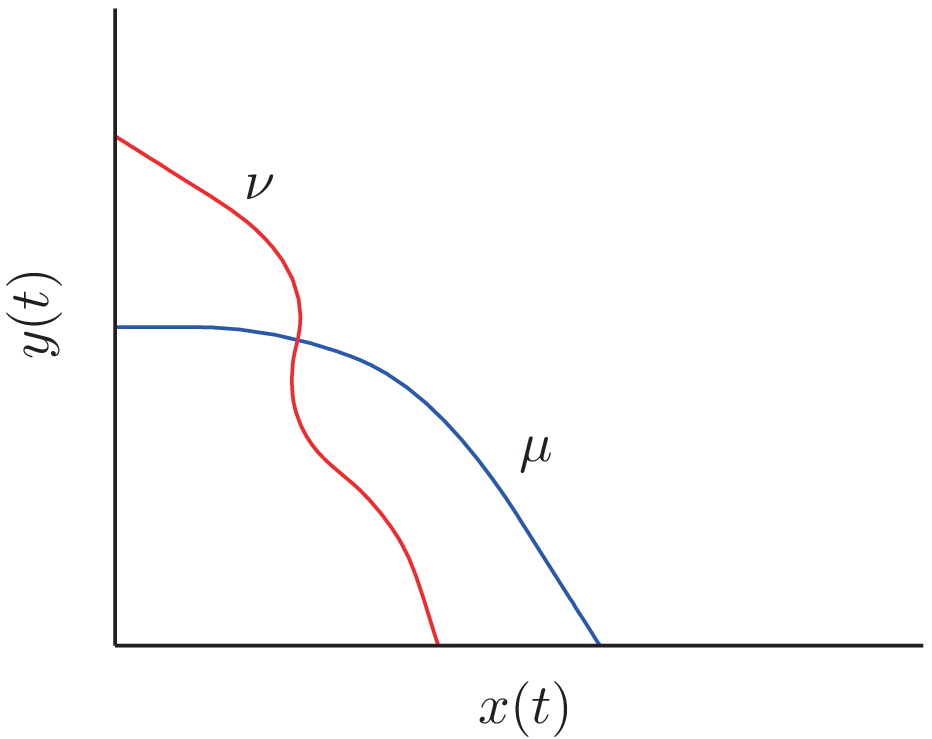}
	(b)
	\includegraphics[width=0.45\textwidth]{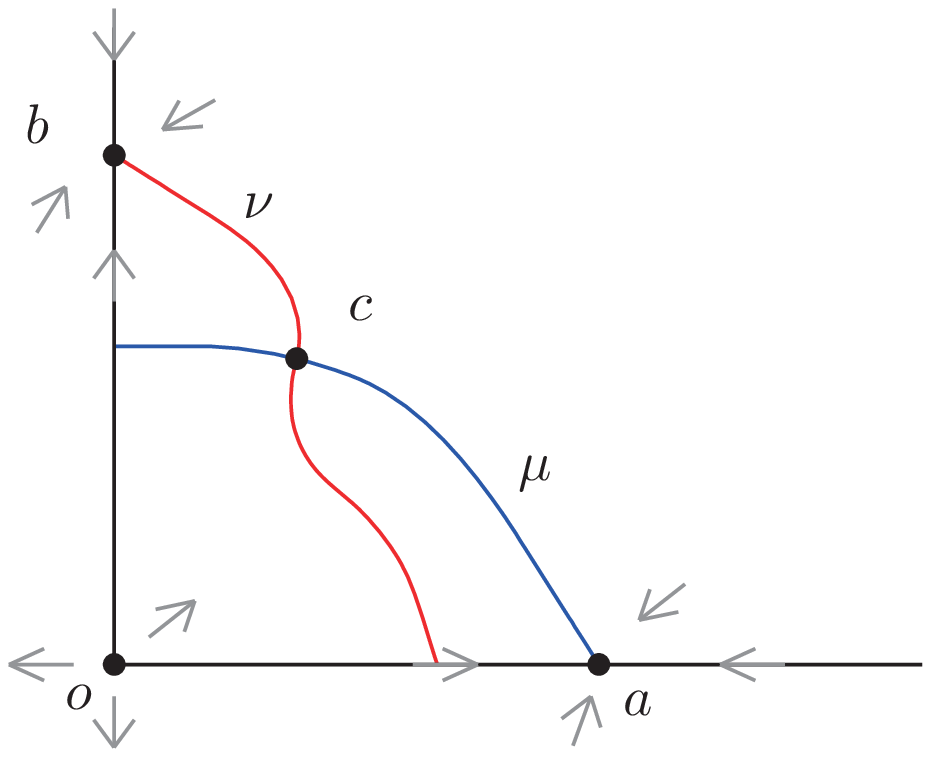}\\
	(c)
	\includegraphics[width=0.45\textwidth]{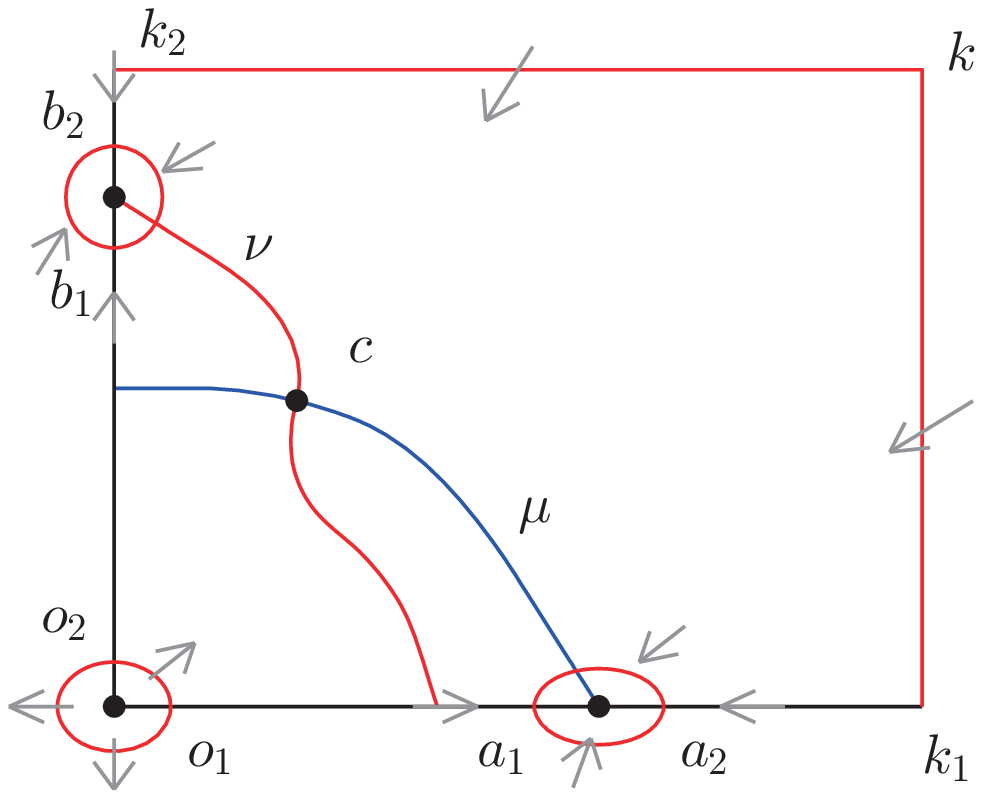}
	(d)
	\includegraphics[width=0.45\textwidth]{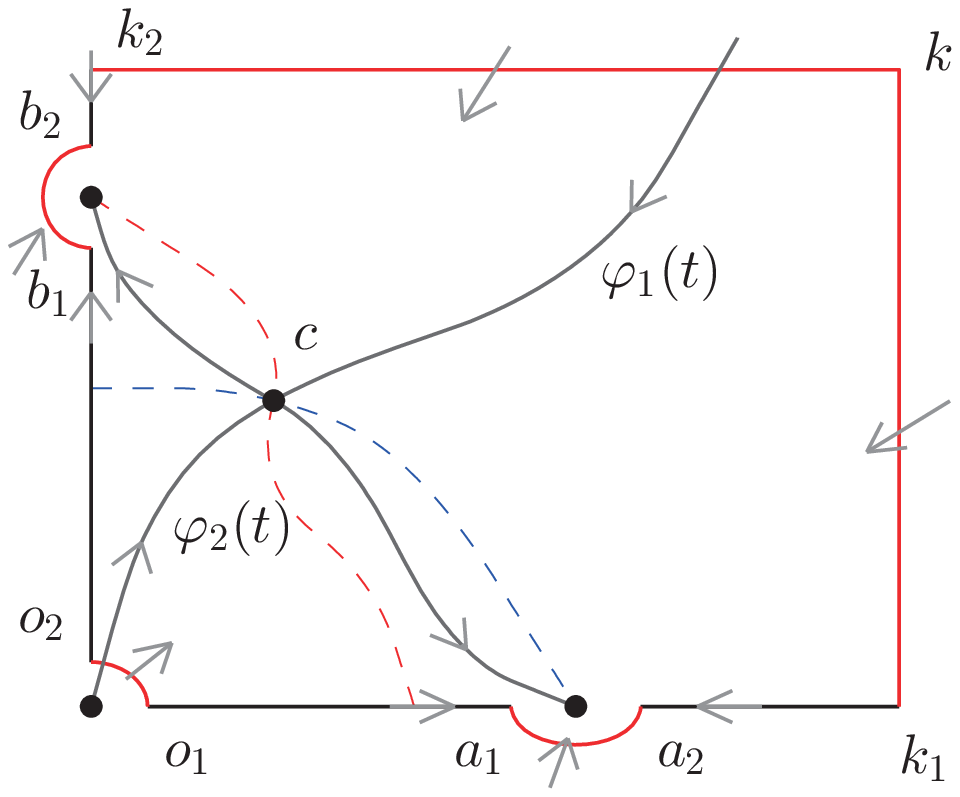}
	\caption{The bistable structure diagram for the competitive model \ref{eq-3-1}.
		(a)The curves $\mu$ and $\nu$ are $x$-nullcline and $y$-nullcline, respectively,
		and they correspond to the equations $M=0$ and $N=0$, respectively.
		(b)There are four equilibria, $a$, $b$, $c$, and $o$, in which $a$ and $b$ are stable
		nodes, and $o$ is an unstable node, and $c$ is unstable.
		(c)The region $U$ required in \textbf{H1$^\prime$} is being constructed.
		(d)The orbits, $\varphi_1$ and $\varphi_2$, together with equilibrium $c$, form the
		separatrix. These two orbits are precisely located on the stable manifold of $c$.}
	\label{figure-4}
\end{figure}

In this case, the model is assumed to be structurally stable, since typically the
eco-system is insensitive to the small disturbance of the external environment, and it is also
impossible to have periodic orbits \cite{Strogatz2018}.  So \textbf{H2-3} is satisfied.
Then by Theorem~\ref{th-2-2}, we obtain that the invariant set $S$, namely the equilibrium point $c$, is a
saddle point. Also, by the comment below Theorem~\ref{th-2-3}, there must be a separatrix, as shown
in Fig.~\ref{figure-4}(d),  which is consists of saddle $c$ and two orbits, $\varphi_1(t)$ and
$\varphi_2(t)$. The orbits to the lower right of the separatrix will flow to equilibrium $a$, while the
orbits on the other side will flow to equilibrium $b$.

Now we consider the competitive model (\ref{eq-3-1}) again, where the nullclines
$\mu$ and $\nu$ are slightly more complicated than the previous one, as shown in
Fig.~\ref{figure-5}(a). The same as before, the origin $o$ is an unstable node,
and equilibrium $b$ is a stable node. However, the differences are that equilibrium $a$
becomes a saddle, and equilibrium $c$ becomes a stable node, which can be obtained
by checking condition (\ref{eq-3-6}). Besides, there are three other equilibria,
$d$, $e$, and $f$, as shown in Fig.~\ref{figure-5}(b).

The construction for region $U$ is also a little different since equilibrium $a$ is
not stable anymore. Because $a$ is a saddle and its stable manifold is on the
positive $x$-axis, there must exist an orbit $\varphi^1(t)$ above the $x$-axis
which comes from the origin and flows towards equilibrium $a$, and then flows upward
along the unstable manifold of $a$, as shown in  Fig.~\ref{figure-5}(c).
Similarly, to the right side of equilibrium $a$, there must exist orbit $\varphi^2(t)$
which flows towards $a$ to the right of $a$ and then flow upward along the unstable
manifold. Near equilibrium $a$, we can find a line segment $\overline{a_1a_2}$, such that
it intersects transversely with $\varphi^1(t)$ at point $a_1$ and with $\varphi^2(t)$
at $a_2$ and with other orbits between them. The point on the $\overline{a_1a_2}$ will leave
the line segment immediately and flow upward. For the other part of the boundary
$\partial U$, the construction method is the same as above. We need to note that,
in the present case, $o_1$ becomes the intersection point of orbit $\varphi^1(t)$ with a closed
curve surrounding $o$, and $k_1$ is the intersection point of orbit $\varphi^2(t)$ with vertical
line segment $\overline{k_1k}$. Now the relevant curves and orbit segments can form a closed
region $U$ required in \textbf{H1$^\prime$}, as shown in  Fig.~\ref{figure-5}(c).

Since there are precisely three unstable equilibria, $d$, $e$, and $f$, it can be
obtained from Theorem~\ref{th-2-2} that these three equilibria must be two
saddles and one repeller. Besides, by analyzing the direction of the vector
field around these three equilibria, we can obtain that there must exist two heteroclinic orbits,
such that both of them come from $e$ and flow towards to $d$ and $f$, respectively.
That is to say, the invariant set $S$ consisting of three equilibria and two heteroclinic orbits is
connected. Thus we can conclude that $e$ is an unstable node, and $d$ and $f$ are saddles.
What is more, by the comment below Theorem~\ref{th-2-3},
there exists a separatrix, which is similar to the above conclusion and is composed
of two orbits, $\varphi_1(t)$ and $\varphi_2(t)$, three equilibria, $d$, $e$, and $f$,
and corresponding connecting orbits. Just like its name, the separatrix divides the
region into two parts, such that the orbits on the boundary will flow into its corresponding
attractor, $b$ or $c$, as shown in  Fig.~\ref{figure-5}(d).

\begin{figure}[tbhp]
	\centering
	(a)
	\includegraphics[width=0.45\textwidth]{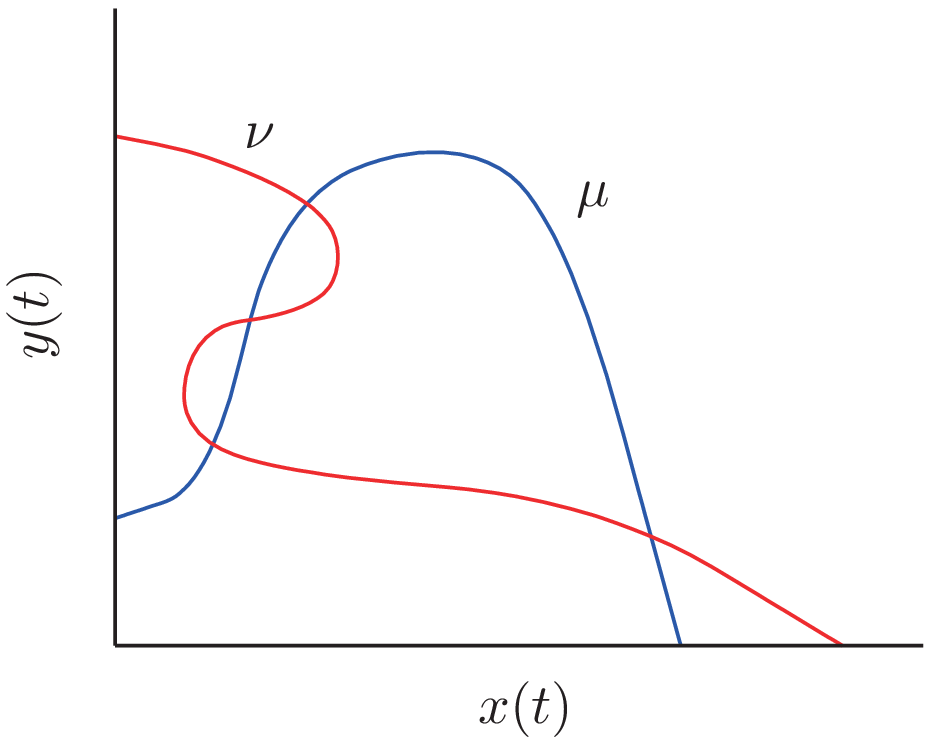}
	(b)
	\includegraphics[width=0.45\textwidth]{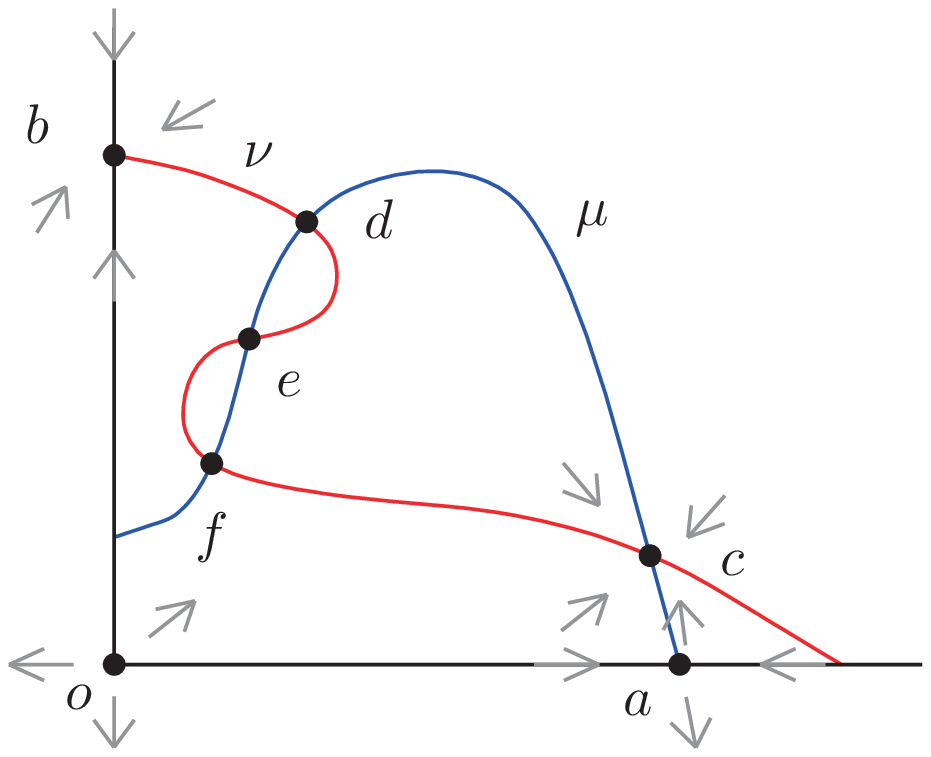}\\
	(c)
	\includegraphics[width=0.45\textwidth]{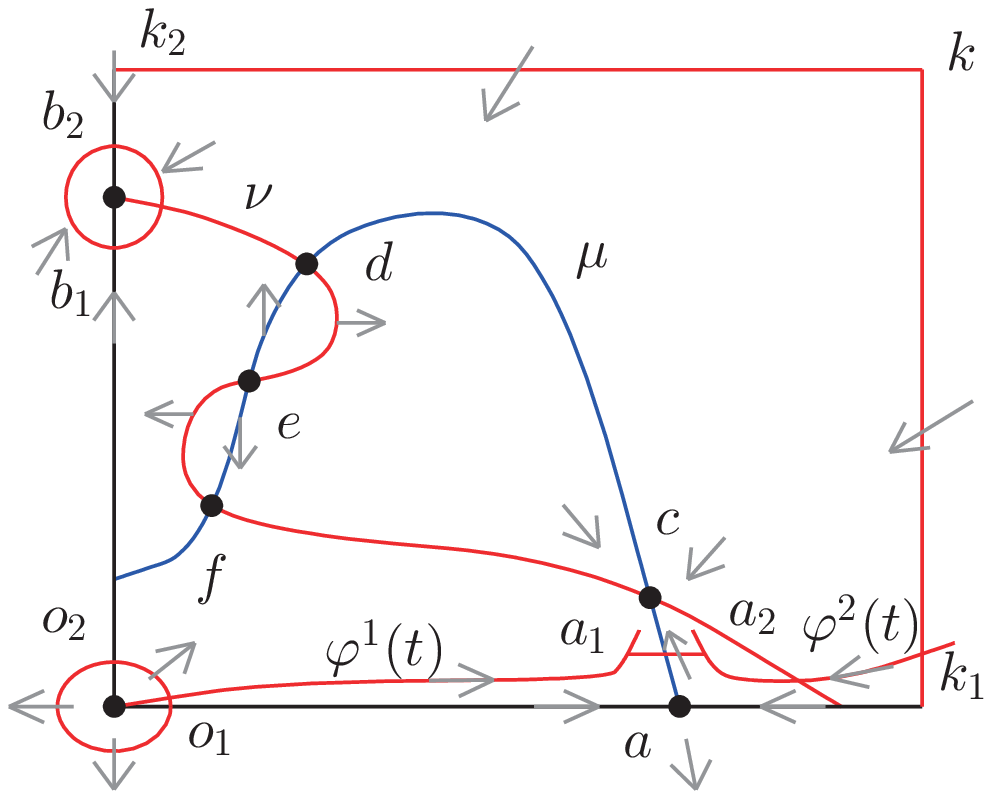}
	(d)
	\includegraphics[width=0.45\textwidth]{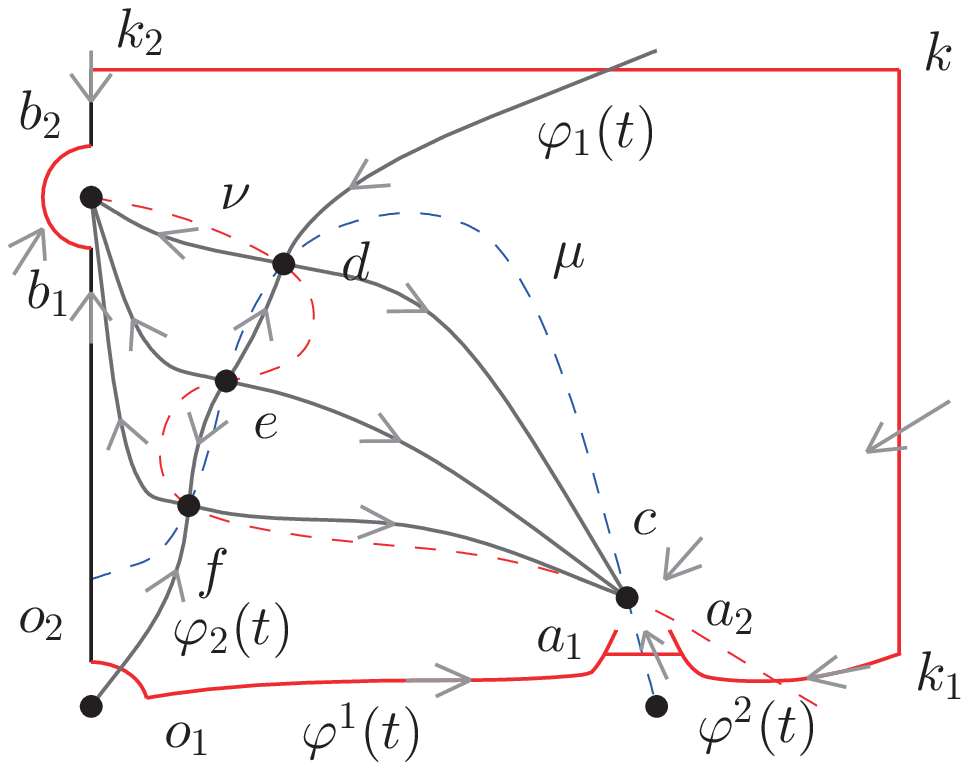}
	\caption{The bistable structure diagram for another competitive model \ref{eq-3-1}.
		(a)The curves $\mu$ and $\nu$ are $x$-nullcline and $y$-nullcline, respectively,
		and they correspond to the equations $M=0$ and $N=0$, respectively.
		(b)There are seven equilibria, in which $o$ is an unstable node, $a$ is a saddle,
		and $b$ and $c$ are stable nodes.
		(c)The region $U$ in \textbf{H1$^\prime$} is being constructed.
		(d)Two orbits, $\varphi_1$ and $\varphi_2$, together with isolated invariant set $S$, form the
		separatrix.}
	\label{figure-5}
\end{figure}

\section{Conclusion and discussion}
\label{sec-4}

Bistability is a widespread phenomenon, whether in daily life or life-related fields. Based on
this observation,  we study two kinds of bistable structures in dynamical systems,
which are very common in the study of biological models. The research method we use is mainly
the Conley index theory, which is a topological tool in dynamical systems. Its
idea is to find an isolated neighborhood first, and then deduce the dynamics of the isolated
invariant set in it by checking the direction of the orbits on the neighborhood boundary.
Besides, if we want to get a more detailed structure of isolated invariant sets, some more
information would be needed. This idea is precisely the strategy of this paper.

In this paper, two kinds of bistable structures with exactly two attractors in an attracting region
are studied. The first is two one-point attractors, and another is the coexistence of a cycle attractor
and a one-point attractor. We prove that in both cases, there are other isolated invariant sets except
two attractors within this attracting region, and there are also connecting orbits from the invariant
set to two attractors, respectively. Besides, if the dynamical system we considered is structurally
stable, we can obtain the requirement for the number of equilibria contained in the invariant set.
More importantly, we find that no matter which case, there is always a separatrix or cycle separatrix
between two attractors. It divides the whole attracting region into two sub-regions so that almost
all the orbits within each sub-region flow to the corresponding attractor. Then, we use the
competitive eco-system to show how to apply the related conclusions to analyze the global
behavior of the model with two attractors. Finally, we briefly analyze four different biological
models to indicate that many models have the bistable structure we considered in this paper,
which are the spruce budworm population model, genetic control model, a predator-prey system
with group defense, and a kind of SIR compartment model with limited medical resources
and supply efficiency.

Finally, there are still challenges related to the bistable structure. The first is the study of
bistable structure in high-dimensional systems because many biological models are built
to be high-dimensional. Thus, the dynamic behavior of orbits here can be very complicated,
which makes it hard to analyze. Besides, researching the bistable structure with a strange
attractor is also tricky, and the strange attractor itself is a tricky object. Also, it is not easy
to study a multistable structure, which is likely to appear in biological models due to the
complexity and diversity of the corresponding biological system.

\section*{Acknowledgements}
\label{sec-5}

This work was jointly supported by the NSFC grants under Grant Nos. 11572181 and 11331009.
We would like to thank Prof. Ping Ao and Dr. Shuang Chen for their helpful advice and discussions.
This work was partly done during Junbo Jia's visit to LAMPS at York University, and he would also
like to thank all members of LAMPS for their support, guidance, and companionship.



\begin{appendix}
	\section{Appendix: Some related known results on the Conley index}\label{sec-appendix}
	
	For the convenience of readers, here we briefly list the definition of the Conley index and related notions.
	For more details about the Conley index, please refer to~\cite{Conley1978,Smoller2012,Mischaikow2002}.

	Let $\varphi$ is a flow on $X$ a locally compact metric space.
	\begin{definition}
		A compact set $N\subset X$ is an \emph{isolating neighborhood} if
		\begin{equation*}
			\mathrm{Inv}(N, \varphi) := \{x\in N | \varphi(\mathbb{R},x) \subset N \} \subset \mathrm{int} N,
		\end{equation*}
		where $\mathrm{int} N$ denotes the interior of $N$. $S$ is an \emph{isolated invariant set} if $S=\mathrm{Inv}(N)$
		for some isolating neighborhood $N$.
	\end{definition}
	
	\begin{definition}\label{def-1}
		Let $S$ be an isolated invariant set. A pair of compact sets $(N, L)$ where $L\subset N$ is called
		an \emph{index pair} for $S$ if:
		\begin{enumerate}
			\item $S=\mathrm{Inv}(\mathrm{cl}(N\setminus L))$ and $N \setminus L$ is a neighborhood of $S$.
			\item $L$ is \emph{positively invariant} in $N$; that is given $x\in L$ and $\varphi([0,t],x)\subset N$, then $\varphi([0,t],x)\subset L$.
			\item $L$ is an \emph{exit set} for $N$; that is given $x\in N$ and $t_1 >0$ such that $\varphi(t_1,x)\notin N$, then there exists $t_0\in [0,t_1]$ for which $\varphi([0,t_0], x)\subset N$ and $\varphi(t_0,x)\in L$.
		\end{enumerate}
	\end{definition}
	
	\begin{definition}
		The \emph{homotopy Conley index} of $S$ is
		\begin{equation*}
			h(S) = h(S, \varphi) \sim (N/L, [L]).
		\end{equation*}
		
	\end{definition}
	
	Where ``$\sim$" denotes homotopy equivalence relation, and $(N/L, [L])$ is pointed space, namely, the points in $L$ are identical.
	We also need to point out that the Conley index in this paper refers to the homotopy Conley index.
	
	In addition, we list some useful properties of the Conley index below as well,
	including three theorems on its applications, which can also show that our results obtained in
	this paper are important applications of the Conley index.
	
	\begin{property}[Wa\.zewski Property]\label{th-1}
		If $N$ is an isolating neighborhood for the isolated invariant set $S$, and if $H(S)\neq \overline{0}$,
		then $S\neq \emptyset$; i.e., $N$ contains a complete orbit.
	\end{property}

	\begin{property}[Wedge Sum]\label{th-2}
		If $S_1$ and $S_2$ are disjoint isolated invariant sets, then the disjoint union $S=S_1 \sqcup S_2$ is
		an isolated invariant set, and
		\begin{equation}\label{eq-1}
			H(S)=H(S_1 \sqcup S_2)=H(S_1)\vee H(S_2).
		\end{equation}
	\end{property}
	
	\begin{property}[Continuation]\label{th-3}
		If $S_\lambda$ and $S_\mu$ are related by continuation, then they have the same Conley index.
	\end{property}
	
	\begin{theorem}\cite{Smoller1972,Smoller2012,Kappos1995}
		Let $\mathbf{x}'=f(\mathbf{x})$ be gradient-like in an isolated neighborhood $N$, and
		let $N$ contain precisely two rest points $x_1$, $x_2$, of $f$, not both of which are
		degenerate. Let $S(N)$ be the maximal invariant set in $N$. If $H(S(N))=\overline{0}$,
		then there is an orbit of $f$ connecting the two rest points.
	\end{theorem}
	
	\begin{theorem}\cite{Conley1971,Kappos1995}
		Let $N$ be an isolating neighborhood of the flow $\phi$ containing precisely two equilibria $p_1$ and $p_2$.
		If $H(S(N)) \neq H(p_1) \vee H(p_2)$, then there exists an orbit $\gamma$ of $\phi$ in $N$, different
		from $p_1$ and $p_2$. If, moreover, $\phi$ is also gradient-like in $N$, then $\gamma$
		connects $p_1$ and $p_2$.
	\end{theorem}
	
	\begin{theorem}\cite{Salamon1985,Kappos1995}
		If $(A_\mu, A_\mu^*)$ is an attractor-repeller pair for the i.i.s. $S_\mu$,
		which continues for $\mu\in [\mu_1,\mu_2]\subset R$ and
		\[S_mu_1=A_{\mu_1} \bigcup A_{\mu_2}^*,S_{\mu_2}=A_{\mu_2} \bigcup A_{\mu_2}^*,\]
		but $CSS(S_{\mu_1})$ is not the same as $CSS(S_{\mu_1})$, then, for
		some $\mu \in [\mu_1,\mu_2]$, there exists a connecting orbit
		from $A_\mu^*$ to $A_\mu$.
	\end{theorem}
	
\end{appendix}

\end{document}